\documentclass[11pt,reqno]{amsart}

\usepackage[margin=1in]{geometry}
\usepackage[colorlinks]{hyperref}
\usepackage[utf8]{inputenc}
\usepackage{amssymb,amsthm,mathtools,mathrsfs,
  enumerate,verbatim,comment,enumitem,amsmath}
\usepackage{thmtools}
\usepackage{cleveref}

\usepackage{xcolor}	
\usepackage{soul}

\usepackage{todonotes}

\usepackage{graphicx}
\usepackage{tikz-cd}
\usepackage{mathabx}

\newtheorem{Theorem}{Theorem}[section]

\newtheorem{Lemma}[Theorem]{Lemma}
\newtheorem{Corollary}[Theorem]{Corollary}

\theoremstyle{definition}
\newtheorem{Definition}[Theorem]{Definition}

\newtheorem{Remark}[Theorem]{Remark}

\DeclareMathOperator{\dom}{\mathrm{dom}}

\newcommand{\rst}{{\restriction}}
\newcommand{\conc}{^{\smallfrown}}

\newcommand{\rca}{\mathsf{RCA}_0}

\newcommand{\prsou}{\mathsf{PRS\omega U}}

\newcommand{\atr}{\mathsf{ATR}_0}
\newcommand{\pica}{\Pi^1_1\mbox{-}\mathsf{CA}_0}

\DeclareMathOperator{\rank}{{rk}}
\DeclareMathOperator{\prim}{{Pr}}

\newcommand{\barprec}{\triangleleft}

\title{Generalized Higman's Theorem and iterated ideals}
\author{Fedor Pakhomov and Giovanni Soldà}

\address{Fedor Pakhomov, Department of Mathematics: Analysis, Logic and Discrete Mathematics, Ghent University, Krijgslaan 281 S8, 9000 Ghent, Belgium, and\newline
Steklov Mathematical Institute of the Russian Academy of Sciences. Ulitsa Gubkina 8, Moscow 117966, Russia}
\email{fedor.pakhomov@ugent.be}

\address{Giovanni Sold\`a, Department of Mathematics: Analysis, Logic and Discrete Mathematics, Ghent University, Krijgslaan 281 S8, 9000 Ghent, Belgium}
\email{giovanni.a.solda@gmail.com}

\thanks{The work of the first author has been funded by the FWO grant G0F8421N. 
The work of the second author has been funded by the FWO senior post-doctoral fellowship 1257725N.
}

\begin{document}

\maketitle

\begin{abstract}
    Generalized Higman's Theorem is the direct counterpart of Higman's Theorem that asserts the closure of the class of \emph{better} quasi-orders, instead of the class of \emph{well} quasi-orders, under the construction $P\mapsto P^{<\omega}$ of the embeddability order on finite sequences. %instead of the class well-quasi-orders asserts the closure of the class of better-quasi-orders (bqo) under constructions of embeddability orders on finite sequences $P\mapsto P^{<\omega}$. 
    Traditionally, this result is obtained as a consequence of very powerful and general techniques of Nash-Williams. In this paper, we propose a new proof of this result that is based on an explicit characterization of the underlying orders. In particular, this new technique allows us to formalize the proof of the result in the formal theory $\atr$, thus resolving a long-standing open problem in the field of reverse mathematics.

    The main ingredient of our proof is the introduction of a transfinite hierarchy of orders $\dot I^*_\alpha(P)$ starting with $\dot I^*_0(P)=P$ and $\dot I^*_{\alpha+1}(P)$ being the inclusion order on ideals of $\dot I^*_{\alpha}(P)$. On one hand, we show that a quasi-order $P$ is a bqo if and only if all $\dot I^*_\alpha(P)$ are wqos. On the other hand, under the assumption that $P$ is a bqo, we show that the $\dot I^*_\alpha(P^{<\omega})$ are wqos, and furthermore give a characterization of their structure in terms of a transfinite iteration of a Higman-like construction. The sufficiently explicit character of this proof allows us to formalize it in a rather straightforward manner.
\end{abstract}

\section{Introduction}

The main aim of this paper is to give a (to the best of our knowledge) new proof of the statement ``if $(P,\leq_P)$ is a wqo, then $(P^{<\omega},\preceq)$ is a bqo'', where $\preceq$ is the usual embeddability relation between sequences: this statement is sometimes referred to in the literature as Generalized Higman's Theorem, and we adopt this convention.
This result was originally implicitly proven by Nash-Williams in \cite{nwt-nash-williams}. He furthermore proved that $(P,\leq_P)$ is a bqo if and only if every set of transfinite sequences on $P$ ordered by embeddability is also a bqo.
We shall refer to this result in the following as Nash-Williams' Theorem.
Other than being an interesting combinatorial statement in itself, the Generalized Higman's Theorem is of particular interest from the point of view of reverse mathematics: indeed, it was shown by Marcone in \cite{nash-williams-marcone} that, over $\atr$, Nash-Williams' Theorem is equivalent to Generalized Higman's Theorem.
In the same paper, Marcone gave a proof of the Generalized Higman's Theorem in $\pica$; the proof was later modified by Towsner in \cite{partial-impred-towsner} to hold in the weaker system $\mathsf{TLPP}_0$. We are not aware of any better upper bound on the strength of the Generalized Higman's Theorem.
A major issue in adapting the classical proofs of Nash-Williams' Theorem is the fact that they rely on (some forms of) the Minimal Bad Array Lemma, which has been shown in \cite{min-bad-freund-pakhomov-solda} to be equivalent to $\Pi^1_2\mbox{-}\mathsf{CA}_0$ over $\atr$.
Indeed, a key ingredient in Marcone's proof of Generalized Higman's Theorem was the identification of a weaker form of the Minimal Bad Array Lemma, namely the so-called Locally Minimal Bad Array Lemma; he showed, however, that this principle is not only implied by but equivalent to $\pica$ over $\rca$.
One can then conclude, as was done in the aforementioned papers, that a proof of Generalized Higman's Theorem in weaker systems would most likely need to be quite different from the existing ones.

The main characters of the proof of Generalized Higman's Theorem we propose are iterated ideals $\dot I^*_\alpha(Q)$ of a qo $Q$, which we introduce in \Cref{sec:defin}.
We remark that the space of ideals of a wqo has already received attention before: most notably, it was shown in \cite{well-better-id-carroy-pequignot} that a qo $P$ is a bqo if and only if it is a wqo and the set of its non-principal ideals (ordered by inclusion) is a bqo, a result known as Pouzet's conjecture.
As the name suggests, the idea behind the iterated ideals is to look at a restriction of the qo given by the iteratively downward-closed sets $\dot V^*_\alpha(Q)$, which we used in \cite{nwfr-pakhomov-solda} (we also refer to \cite{3-bqo-freund}, where essentially the same sets were introduced).
In \Cref{sec:it-id-bqo}, we show that the iterated ideals still encode enough information about the order $Q$ to determine whether the latter is a bqo or not.
Namely, we show that $Q$ is an $\alpha$-wqo if and only if $\dot I^*_\alpha(Q)$ is a wqo.

It is therefore sufficient to show that if $P$ is a bqo, then $\dot I^*_\alpha(P^{<\omega})$ is a wqo for every $\alpha$. Firstly we reduce the wqo-ness of $\dot I^*_\alpha(P^{<\omega})$ to the wqo-ness of $\prim(\dot I^*_\alpha(P^{<\omega}))$, where $\prim(\dot I^*_\alpha(P^{<\omega}))$ is the set of primes with respect to the multiplication operation on $\dot I^*_\alpha(P^{<\omega})$ which is the canonical extension of the concatenation operation on $P^{<\omega}$ (see \Cref{sec:vq-iq-mult-qo}). After this reduction, in order to finish the proof our remaining goal is to construct order-reflecting maps from $\prim(\dot I^*_\alpha(P^{<\omega}))$ into $\dot V^*_\alpha(P\sqcup\{\star\})$; here, $\star$ is a fresh element incomparable with everything else, and the fact that all $\dot V^*_\alpha(P\sqcup\{\star\})$ are wqos is implied by the fact that $P$ is a bqo (see \cite{wq-trees-nash-williams,phd-pequignot,fraisee-conj-laver,nwfr-pakhomov-solda}). The construction of the maps is performed by transfinite recursion on $\alpha$, and is presented in \Cref{sec:main-theo}, while the structural characterization of $\prim(\dot I^*_{\alpha+1}(P^{<\omega}))$ in terms of $\prim(\dot I^*_\alpha(P^{<\omega}))$ crucially required for the successor step of the recursion is a special case of \Cref{theo:id-two-forms}, which we prove in \Cref{sec:id-combinatorics}.

Let us say a bit more about the results of \Cref{sec:id-combinatorics}: following the fundamental insight of Goubault-Larrecq et al. \cite{ideal-decomp-goubault-larrecq-halfon-karakndikar-narayan-schnoebelen} as well as Kabil and Pouzet \cite{id-higman-kabil-pouzet}, we show that, provided $Q$ is a monoid and a wqo under some extra natural conditions, the prime ideals of $Q$ are generated from the primes of $Q$ in just two ways. In \Cref{sec:concl} we discuss the relation between our construction and earlier developments of \cite{ideal-decomp-goubault-larrecq-halfon-karakndikar-narayan-schnoebelen,id-higman-kabil-pouzet} in more detail.

In the second part of \Cref{sec:vq-iq-mult-qo}, we deal with some technical details related to the formalization of these kinds of results in weak theories, and give some additional properties of primes that contribute to the construction in \Cref{sec:main-theo}.
Putting all this together, we can conclude that if $P$ is a bqo, so is $P^{<\omega}$, which we do in \Cref{cor:ght-atr}, where we also remark that $\atr$ suffices for this proof.
This proves that Generalized Higman's Theorem (and therefore Nash-Williams' Theorem) is equivalent to $\atr$ over $\rca$: the reversal is given by Shore's results in \cite{fraisse-conj-shore}, together with a small correction due to Freund and Manca in \cite{freund-manca-wwo}.
These results answer the second part of Question $24$ of \cite{open-questions-montalban}.

In the last section, \Cref{sec:comparing-el-gen-hig}, we establish a finer characterization of the orders $\dot I^*_\alpha(P^{<\omega})$ that is not required for the proof of our main result. Namely, we show that the orders $\dot I^*_\alpha(P^{<\omega})$ are isomorphic to what we call generalized Higman orderings: a certain variant of the embeddability ordering on finite sequences. 

In this paper, we prove our results in the theory $\prsou(\mathsf{ch})$, which is a rather weak system of set theory, where operations on sets should be given by primitive-recursive set functions relativized to a global choice function. On one hand, this theory allows us to develop relevant mathematical constructions in a very familiar set-theoretic fashion, with some extra care to make some definitions sufficiently ``constructive''. On the other hand, $\prsou(\mathsf{ch})$ admits an interpretation in $\atr$, thus allowing us to draw reverse-mathematical conclusions. Here we broadly follow a framework similar to our preceding paper \cite{nwfr-pakhomov-solda}, with the main difference being that the theory $\prsou(\mathsf{enu})$ used there included a global enumeration function $\mathsf{enu}$ that witnessed the countability of all the sets, instead of the milder inclusion of the global choice function in this paper.

\subsection*{Acknowledgments} 
We are very grateful to Andreas Weiermann, who pointed us to the paper \cite{ideal-decomp-goubault-larrecq-halfon-karakndikar-narayan-schnoebelen}, which was crucial for the development of the technique of the present paper. We thank Alakh Dhruv Chopra for several instructive conversations related to the questions that we study in \Cref{sec:it-id-bqo}.

\section{Basic definitions}\label{sec:defin}

The theory we will mostly be using for this paper is an extension of $\prsou$, i.e.\ primitive recursive set theory with urelements and $\omega$. We introduced this theory in \cite{nwfr-pakhomov-solda}, but we recommend the presentation given in \cite{phd-freund} for an introduction to it: there, essentially the same theory without urelements is presented, and the addition of urelements does not change much.

We are interested in extending the theory $\prsou$ with a choice function $\mathsf{ch}$, i.e.\ a function satisfying the axioms
\begin{itemize}
    \item $(x\neq \emptyset \wedge \neg\mathsf{U}(x)) \rightarrow \mathsf{ch}(x)\in x$, and
    \item $(x=\emptyset \vee \mathsf{U}(x)) \rightarrow \mathsf{ch}(x)=\emptyset$.
\end{itemize}
Of course, we furthermore allow $\mathsf{ch}$ to be used to define new functions using primitive set recursion and compositions.

We shall refer to this extension of $\prsou$ as $\prsou(\mathsf{ch})$.
We remark that the existence of such a function would be implied by the axiom of countability, and therefore $\prsou(\mathsf{ch})$ is weaker than the theory $\prsou(\mathsf{enu})$ used in the previous paper, which essentially extends $\prsou$ with the axiom of countability. The essential reason why we need a choice function is because we will need, at several points, to exploit a well-known property of wqos, namely that an upward-closed subset of a wqo only has finitely many minimal elements, which we need to find. The theory $\prsou$ alone does not seem capable of performing this operation.

%As done in \cite{nwfr-pakhomov-solda}, we work in the theory $\prsou$ of primitive recursive set theory with $\omega$ and urelements. 
For the rest of this section, we fix the language and the structure of urelements as, respectively, $\{\leq_Q\}$ and $(Q,\leq_Q)$, where $\leq_Q$ is a binary relation and $(Q,\leq_Q)$ is a quasi-order. We remark now that we will later expand the language with a function symbol, but the content of this section remains valid. 

In general, we use the notation $\equiv_Q$ for the equivalence relation induced by the pre-order $\leq_Q$ on $Q$, or simply $\equiv$ if the order $\leq_Q$ is sufficiently clear from the context, with one exception: we denote by $\sim^*_Q$ the equivalence relation induced by $\lesssim^*_Q$ (which we define below).

We now recall some fundamental definitions from \cite{nwfr-pakhomov-solda}: we refer to Section 4 of that paper for a more thorough discussion on the concept we now introduce.

\begin{Definition}\label{def:lesssim}
    ($\prsou$)
        We define a rank class function $\rank$ so that it is undefined on the class $Q$ of urelements, and 
    \[
    \rank(x)\leq\alpha \Leftrightarrow \forall y\in x(y\not\in Q \rightarrow \mathsf{rk}(y)<\alpha).
    \]  
  
    Let $\dot V(Q)$ consist of all elements of $V(Q)$ that are either urelements or contain an urelement in their transitive closure. Let $\dot V_{\alpha}(Q)$ be the subclass of $\dot V(Q)$ consisting of elements of rank $<\alpha$. 
    
  The pre-order $x \lesssim_Q^* y$ on $\dot V(Q)$ is defined by primitive recursion as follows:
  $$x\lesssim_Q^* y \iff \begin{cases} x\leq_Q y & \text{if }x,y\in Q\\ x\lesssim_Q^* \{y\} &\text{if }y\in Q \text{ and }x\not\in Q\\ \{x\}\lesssim_Q^*y  &\text{if }x\in Q \text{ and }y\not\in Q \\ \forall x'\in x\exists y'\in y(x'\lesssim_Q^*y') &\text{if }x,y\not\in Q\end{cases}$$ 
We denote by $\dot V^*(Q)$ and $\dot V^*_\alpha(Q)$ the quasi-orders $(\dot V(Q),\lesssim^*_Q)$ and $(\dot V_\alpha(Q),\lesssim^*_Q)$.
\end{Definition}

In this paper, a suborder of $\dot V^*(Q)$ will play a crucial role: the order of the \emph{iterated ideals}. Intuitively, if $\dot V^*(Q)$ can be seen (over a sufficiently strong theory) as being generated by iterating applications of the map $X\mapsto Q\cup \dot{\mathcal{P}}(X)$ (where $\dot{\mathcal{P}}$ is the operation returning the powerset minus the empty set), the order of iterated ideals is obtained by restricting $\dot{\mathcal{P}}$ to only produce sets that are upward-directed with respect to $\lesssim^*_Q$.

\begin{Definition}
  $(\prsou)$ 
  For every ordinal $\alpha$, we define the classes $\dot I^*_\alpha(Q)$ of elements of $\dot V^*_\alpha(Q)$ that are hereditarily upward-directed with respect to $\lesssim^*_Q$. Formally, by recursion on $\alpha$ we define:
  \[
    \dot I^*_\alpha(Q) = Q\cup \{ x\in \dot V^*_\alpha(Q):
    \forall x'\in x\exists \beta<\alpha (x'\in \dot I^*_\beta(Q))
    \text{ and } \forall y,z\in x \exists w\in x (y,z \lesssim_Q^* w) \}
  \]
  As standard, we abuse notation and denote by $\dot I^*_\alpha(Q)$ the qo
  $(\dot I^*_\alpha (Q),\lesssim_Q^*)$.

  We denote by $\dot I^*(Q)$ the subclass of $\dot V^*(Q)$ formed by elements that belong to some $\dot I^*_\alpha(Q)$.
%  \todo{modify this}
\end{Definition}

$\dot I^*_\alpha(Q)$ is clearly a suborder of $\dot V^*_\alpha(Q)$. It is easy to find examples where it is strictly smaller: for instance, if $Q$ contains two incomparable elements, then $\dot I^*_\alpha(Q)$ will also consist of just two incomparable elements for every $\alpha$, whereas $\dot V^*_\alpha(Q)$ consists of two incomparable elements and a chain of order type $\alpha$ above them.
It is equally easy (although admittedly not very informative) to find cases in which $\dot I^*_\alpha(Q)= \dot V^*_\alpha(Q)$ for every $\alpha$: namely, it is sufficient to let $Q$ be any linear order and to prove by simultaneous induction on $\alpha$ that all $\dot V^*_\alpha(Q)$ are linearly ordered by $\lesssim_Q^*$ and $\dot I^*_\alpha(Q)= \dot V^*_\alpha(Q)$.

\section{Iterated ideals and bqo-ness}\label{sec:it-id-bqo}

The first result is that iterated ideals code enough information to
determine if $Q$ is a bqo or not. The proof of the following Lemma is similar to that of \cite[Proposition 4.14]{nwfr-pakhomov-solda}, but with one important additional trick, heavily inspired by the proof of the Clopen Ramsey Theorem in $\atr$ given in \cite{book-simpson}.

All the terminology is identical to the one used in \cite{nwfr-pakhomov-solda}. We recall the relevant definitions for the reader's convenience, and because the choice of working in a weak theory like $\prsou$ implies that the definition of a front cannot be the standard one (namely, we need fronts to be equipped with a rank).

\begin{Definition}
The following definitions are given over $\prsou$.
\begin{itemize}
    \item A \emph{front} is a pair $(F,\mathsf{rk}_F)$ 
  where $F\subseteq [\omega]^{<\omega}$ is an infinite $\sqsubseteq$-antichain
  such that, for every $X\in [\bigcup F]^{\omega}$, there is $\sigma\in F$
  such that $\sigma\sqsubseteq X$, 
  and $\mathsf{rk}_F$ is a function from
  $$T_F=\{\sigma \in \bigl[ \bigcup F \bigr]^{<\omega}\mid
  \exists \tau \in F( \sigma\varsqsubsetneq \tau)\}$$
  to ordinals such that $\mathsf{rk}_F(\sigma)>\mathsf{rk}_F(\tau)$ for
  $\sigma\varsqsubsetneq\tau$ in $T_F$. 
  \item We say that the \emph{rank of $(F,\mathsf{rk}_F)$ is $\le \alpha$} if
  $\mathsf{rk}_F(\emptyset)\le\alpha$, and that it is $<\alpha$ if
  $\mathsf{rk}_F(\emptyset)<\alpha$.
  \item For non-empty $\sigma,\tau\in[\omega]^{<\omega}$ we put
  $$\sigma\triangleleft \tau  \stackrel{\mbox{\scriptsize\textrm{def}}}{\iff}\min(\sigma)<\min(\tau) \text{ and }\sigma \setminus \{ \min \sigma\} \text{ is $\sqsubseteq$-comparable with }\tau.$$
  \item An \emph{array} is a function $f\colon F\to Q$ from a front $F$ to a quasi-order $Q$. We say that $f$ has \emph{rank} $\le\alpha$ if $F$ has rank $\le \alpha$. We say that $f$ is \emph{bad} if $f(\sigma)\not\le_Q f(\tau)$ holds for every $\sigma,\tau\in F$ such that $\sigma\triangleleft \tau$. We say that $f$ is \emph{good} if it is not bad.
  \item   We say that a quasi-order $Q$ is an \emph{$\alpha$-well quasi-order} (henceforth \emph{\emph{$\alpha$}-wqo}) if it is well-founded (i.e., every non-empty subset of it has a minimal element) and there are no bad arrays $f\colon F\to Q$ of rank $<\alpha$.
  \item We say that that $Q$ is a \emph{better quasi-order} (henceforth \emph{bqo}) if it is an $\alpha$-wqo for every $\alpha$ (or equivalently, if there are no bad arrays whose codomain is $Q$).
\end{itemize}
\end{Definition}

\begin{Lemma}
  $(\prsou)$ Suppose that there are no bad sequences $f:\omega\to \dot I^*_\alpha(Q)$,
  then there are no bad arrays $g:F\to Q$ of rank $\leq\alpha$.
\end{Lemma}
\begin{proof}
 % The argument is essentially a simplified version of \cite[V.9.4]{book-simpson}.

  We proceed by contradiction: suppose that $g:F\to Q$ is a bad array of rank $\leq\alpha$. We will use it to build a bad sequence $f$ into $\dot I^*_\alpha(Q)$.

  % Let $T^F=\{ \sigma\in [\omega]^{<\omega} : \exists \tau\in F(\sigma\sqsubseteq \tau)\}.$
  % Notice that $F$ of the rank $\leq\alpha$ implies that the order type of
  % $(T^F,\leq_{KB})$ is at most $\omega^{\alpha+1}$. Hence, we can proceed by recursion
  % on the order-type of $(T^F,\leq_{KB})$. It is immediate to modify $\rank T_F$ to be a
  % rank function on $T^F$

  Notice that $F$ having rank $\leq\alpha$ implies that the order type of
  $(T_F\cup F,\leq_{KB})$ is at most $\omega^{\alpha+1}$: hence, we can proceed by
  recursion on the order type of $(T_F\cup F,\leq_{KB})$. The aim of the recursion
  is to build a function $h:[\bigcup F]^{<\omega}\to \dot I^*_\alpha(Q)
  \times [\bigcup F]^{\omega}$.

  For every $\sigma\not\in T_F\cup F$, let $\tau$ be the
  unique string such that $\tau\in F$ and $\tau\sqsubset \sigma$; we set
  $h(\sigma)= (g(\tau),\bigcup F)$.

  We now start describing the recursion. The minimal element of $(T_F\cup F,\leq_{KB})$
  is an element of $F$, say $\sigma$, and we set $h(\sigma)=(g(\sigma),\bigcup F)$ ($\leq_{KB}$ denotes the Kleene-Brouwer ordering).

  At the start of every other stage, when we want to define $h(\sigma)$, we
  suppose that we have defined $h(\tau)$ for all $\tau\in T_F\cup F$ with
  $\tau<_{KB}\sigma$, say that $h(\tau)= (x_\tau, U_\tau)$. We further suppose that
  $x_\tau\in \dot I^*_{\mathsf{rk}_F \tau +1}(Q)$ if $\tau\in T_F$, $x_\tau\in Q$ if $\tau\in F$,
  and that $\tau_0<_{KB} \tau_1$ implies $U_{\tau_1}\subseteq^\infty U_{\tau_0}$,
  where $A\subseteq^\infty B$ means that $A\setminus B$ is finite.

  Let $K\in [\bigcup F]^\omega$ be such that $K\subseteq^\infty U_\tau$ for every
  $\tau<_{KB} \sigma$, with $\min K> \max \sigma$ (this set can be built easily by recursion on the order type of $(\{\tau\in T_F\cup F: \tau<_{KB} \sigma\},\leq_{KB})$). If $\sigma\in F$, then we set
  $U_\sigma=K$ and $h(\sigma)= (g(\sigma), U_\sigma)$. If instead $\sigma\in T_F$,
  notice that for every $n\in K$,
  $h(\sigma\conc n)$ has already been defined, and
  $\pi_0 (h(\sigma\conc n))\in \dot I^*_{\mathsf{rk}_F \sigma}(Q)$. Consider the sequence
  $e: K \to \dot I^*_{\mathsf{rk}_F \sigma}(Q)$ such that 
  $e(n)= \pi_0(h(\sigma\conc n))$.
  Since $\rank \sigma\leq\alpha$, $\dot I^*_{\mathsf{rk}_F \sigma}(Q)$ is a wqo, and so the sequence $e$ has
  a weakly ascending subsequence: let $U_\sigma$ be the domain of this weakly ascending subsequence.
  Hence, $e(U_\sigma)$ is upward-directed, and so is an iterated ideal: we put
  $x_\sigma= e(U_\sigma)$, and set $h(\sigma)=(x_\sigma,U_\sigma)$ as we just defined. It is easy to see that the assumptions we made to be able to carry out the recursion at the end of the previous paragraph are satisfied.  

  Let $H= U_{\emptyset}$, and define $f:H\to \dot I^*_\alpha(Q)$ as $f(n)=\pi_0(h(\{n\}))$.
  We verify that it is a bad sequence. For simplicity, we define $h'(\sigma)=
  \pi_0(h(\sigma))$ for every $\sigma\in [\bigcup F]^{<\omega}$. We will show the
  following claim: for every two strings $\sigma,\tau\in [H]^{<\omega}$ such that
  $\sigma\barprec \tau$, $|\sigma|=|\tau|$ and $h'(\sigma)\lesssim^*_Q h'(\tau)$,
  there are infinitely many $n\in H$ such that $h'(\sigma\cup\tau)
  \lesssim^*_Q h'(\tau \cup \{n\})$.

  If we prove the claim above, it follows that $f$ is bad. Indeed, suppose it was not,
  as witnessed by $m<n$. Since $\subseteq^\infty$ is a transitive relation,
  $H\subseteq^\infty U_\sigma$ for every $\sigma\in T_F\cup F$. Hence, using the claim, we
  can find $r\in H\cap U_{\{n\}} \cap U_{\{m\}}$ such that $h'(\{m,n\})\lesssim^*_Q
  h'(\{n,r\})$. Proceeding like this, we can find an infinite subset $X\subseteq H$
  such that, for every $l\in\omega$, $h'(X\rst_l)\lesssim^*_Q h'(X^-\rst_l)$, which
  contradicts the assumption that $g$ is bad.

  So we only have to prove the claim. Saying that $h'(\sigma)\lesssim^*_Q h'(\tau)$ means that
  for every $x\in h'(\sigma)$ there is $y\in h'(\tau)$ with $x\lesssim^*_Q y$, which
  in particular is true for $x=h'(\sigma\cup \tau)$. By the fact that 
  $h'(\tau)$ is a weakly ascending sequence, for cofinitely many elements $y$ of $h'(\tau)$ we have $x\lesssim^*_Q y$:
  since every one of these elements is $h'(\tau\cup\{n\})$ for some $n\in U_\tau$, 
  and since $H\subseteq^\infty U_\sigma$, we conclude that there are infinitely many $n$
  as above, as we wanted.
\end{proof}

%Although it is not technically true that $\dot I^*_\alpha(Q)$ is obtained from 
%$\dot I^*_\alpha(Q)$ by adding comparisons,
%the same proof as above shows that if there are no bad
%sequences $f:\omega\to \dot I_\alpha(Q)$, then there are no bad arrays $g:F\to Q$
%where $F$ is of the rank $\leq \alpha$.

The previous Lemma immediately yields the desired result.

\begin{Theorem}\label{theo:ideal-wqo-qo-bqo}
    ($\prsou$) For every $\alpha$, $\dot I^*_\alpha (Q)$ is a wqo if and only if there are no bad arrays $g:F \to Q$, where $F$ is a front of rank $\leq\alpha$.
\end{Theorem}
\begin{proof}
    The previous Lemma gives the left-to-right implication. For the other one, notice that if $\dot I^*_\alpha (Q)$ is not a wqo, then neither is $\dot V^*_\alpha(Q)$, and the result follows from \cite[Proposition 4.12]{nwfr-pakhomov-solda}.
\end{proof}

\section{Multiplicative and monoidal (w)qos}\label{sec:id-combinatorics}

% In this Section, as anticipated, we extend the language $L$ of the structure of urelements with a binary function $\cdot$, which we shall call product: we will thus be working with triples $(Q,\leq_Q,\cdot)$.    
In this section, we study orders with a binary associative operation on them. We remark that, for this section only, the class $Q$ is not regarded as the class of urelements. The reason for doing so is the following: in the first part of the section, we will prove several properties of structures $(Q,\leq_Q,\cdot)$, which we then want to also hold for some other, more complicated entities, such as $(\dot{\mathcal{I}}(Q),\subseteq,\bullet)$. Since it is not the case that both structures can be urelements, we prefer to give the proofs in a more general framework. We will thus leave the class of urelements unspecified in this section.

We remark that this is, by and large, an aesthetic matter, and no real harm ensues if $Q$ was regarded as the class of urelements. 

As we stated in the Introduction, the main result of this section was inspired by \cite{ideal-decomp-goubault-larrecq-halfon-karakndikar-narayan-schnoebelen}. Furthermore, we point out that a closely related line of work was carried out in \cite{id-higman-kabil-pouzet}, where the ideal space over an ordered monoid is studied.
%: in the mentioned work, Kabil and Pouzet prove that (using the terminology we are about to introduce) for every monoidal wqo $Q$ the prime ideals are just ideals of $Q$ or multiplicatively closed ideals. This result, which we found after finishing the present paper, if proved in the appropriate base theory, could probably also be used for the results in the rest of this document. We remark that our main results in this Section, anyway, are neither implies nor imply the ones by Kabil and Pouzet: this is due to the fact that we also show that the $+$-property is passed to the space of ideals.

\begin{Definition}[$\prsou$]
Let $Q$ be a class, $\leq_Q$ be a primitive set-recursive, binary, reflexive and transitive relation, and $\cdot$ be a binary function on $Q$.
  We say that $(Q,\leq_Q,\cdot)$ is a \emph{multiplicative qo} if $(Q,\leq_Q)$ is a qo and
  $(Q,\cdot)$ is such that
  \begin{itemize}
  \item $\cdot$ is \emph{associative with respect to the equivalence} $\equiv$: for every $p,q,r\in Q$, $(p\cdot q)\cdot r \equiv p\cdot (q\cdot r)$.
  \item $\cdot$ is \emph{weakly increasing}: for every $q,p\in Q$, $q\leq_Q q\cdot p$.
  \item $\cdot$ is \emph{monotone}: for every $q,p,q',p'\in Q$ with
    $q\leq_Q q'$ and $p\leq_Q p'$, $q\cdot p \leq_Q q'\cdot p'$.
  \end{itemize}
  $(Q,\leq_Q,\cdot)$ is a \emph{multiplicative wqo} if $(Q,\leq_Q)$ is a wqo, and it is a
  \emph{multiplicative wfqo} if $(Q,\leq_Q)$ is a well-founded qo.

  $(Q,\leq_Q,\cdot)$ is a \emph{monoidal qo} if it is a multiplicative qo and $(Q,\cdot)$ is a monoid (i.e., it has a neutral element). \emph{Monoidal wqos} and \emph{monoidal wfqos} are defined in the obvious way. 
\end{Definition}

%From now on, we shall assume that the structure of the urelements forms a multiplicative qo. 
As customary, we will often omit $\cdot$ and denote $q\cdot p$ by $qp$.

\begin{Remark}
We list some easy observations on the previous Definition.
\begin{itemize}
    \item We notice that every multiplicative qo $(Q,\leq_Q,\cdot)$ is an ordered semigroup if $\leq_Q$ is a partial ordering. It is clearly not true that any ordered semigroup is a multiplicative qo (e.g., if $x\leq_Q y$ implies $x=y$, then $(Q,\leq_Q,\cdot)$ is an ordered semigroup, but not a multiplicative qo in general).
    \item In a monoidal qo, the neutral elements form the minimal $\equiv$-equivalence class of the order: indeed, by weak increasing-ness, for every $q\in Q$ and a neutral $e$ we have $e\leq_Q qe= q$. Though in multiplicative qos, the property of an element to be in the minimal $\equiv$-equivalence does not entail that the element is neutral, e.g. $(\mathbb{N}\setminus\{0\},\le,+)$ has $1$ as its least element, but $1$ is not neutral.%It is clearly not sufficient to be the minimal element in order to be a neutral element (e.g., consider $\omega\setminus \{0,1\}$ with the usual ordering and the usual product).
    \item The prototypical example of a multiplicative qo (and indeed, of a monoidal qo) $(Q,\leq_Q,\cdot)$ to have in mind is the case where $Q$ is $P^{<\omega}$, the class of finite sequences over a qo $P$, $\leq_Q$ is the usual embedding $\preceq$ of sequences, and $\cdot$ is the operation of concatenation of strings $\conc$.
    %\item In the following, we will be working with multiplicative qo's that enjoy another property, that we introduce in \Cref{def:+-prop}: we prove in \Cref{lem:neutral-el-wqo} that, in the presence of that extra property, then the difference between multiplicative wfqo and monoidal wfqo essentially disappears. 
    \item Most of the results we give in the rest of the paper use the assumption that $Q$ is a monoidal qo: this is in order to have a nicer definition of the prime elements. However, as we will see, most of the multiplicative qos that we will be working with will also enjoy the $+$-property (to be introduced in \Cref{def:+-prop}): we show in \Cref{lem:neutral-el-wqo} that every well-founded multiplicative qo with this additional property is also automatically a monoidal qo, so there is in essence no loss of generality in restricting to the latter class.
\end{itemize}
\end{Remark}

%We remark that most of the results we will prove in this Section would carry through if we only required $(Q,\cdot)$ to be a semigroup (i.e., if we did not require that a neutral element exists). We choose not to do that for essentially two reasons: the first is that the intended meaning of the product is the concatenation operation on the Higman ordering $P^{<\omega}$ on a qo $P$, which clearly is a monoid (the empty string being the neutral element). The second reason is that, in any case, we shall assume that the qo $(Q,\leq_Q)$ is indeed a monoidal wqo$^+$ (as defined below): under this assumption, as shown in \Cref{lem:neutral-el-wqo}, $Q$ has a minimal element, which is also the neutral element. The proof would work also with the assumption that $(Q,\cdot)$ is a semigroup instead of a monoid.

% An important consideration for the following is that, contrary to \cite{nwfr-pakhomov-solda}, here we will have to assume that the class of urelements is actually a set. The importance of this assumption will be seen at several points, the first of which is the following Definition.

Although the previous Definition was only assuming that $Q$ was a class, from now on we will actually need $Q$ to be a set: this will have repercussions in the following sections (where $Q$ will be regarded as the class of urelements again), and be an important difference with respect to \cite{nwfr-pakhomov-solda}, where it was not necessary for the urelements to form a set.

\begin{Definition}
    ($\prsou$) 
    Let $Q$ be a set and
    let $(Q,\leq_Q,\cdot)$ be a monoidal qo.
  We say that $p\in Q$ is \emph{prime} if $p$ is not equivalent to the identity and $p \equiv ab$ implies $p\equiv a$
  or $p\equiv b$. We denote by $\prim(Q)$ the set of all prime elements of $Q$.
\end{Definition}

The function $\prim$ would not be primitive set-recursive without the assumption of $Q$ being a set. 

\begin{Lemma}\label{lem:fin-fact}
        ($\prsou$)
  If $Q$ is a set and $(Q,\leq_Q,\cdot)$ is a monoidal wfqo, then every $q\in Q$ is equivalent to a finite
  product of primes, where the product of zero primes is set to be a neutral element.
\end{Lemma}
\begin{proof}
  Suppose not: since $Q$ is a set and $(Q,\leq_Q)$ is well-founded, by $\Delta_0$-induction we can find a $c\not\equiv e$ $\leq_Q$-minimal such that it is not prime
  but it is not expressible as a finite product of primes.
  Since $c$ is not prime,
  there are $a,b\in Q$ such that $c\equiv ab$ and $c$ is not equivalent to
  either $a$ or $b$. Hence $a<_Qc$ and $b<_Q c$, so by definition of $c$
  $a\equiv p_0\dots p_m$ and $b\equiv q_0\dots q_n$. Hence
  $c\equiv p_0\dots p_mq_0\dots q_n$, contradiction.
\end{proof}

We will be interested in multiplicative qos that enjoy a strong property, which we now introduce. This is again inspired by what happens when $(Q,\leq_Q,\cdot)$ is $(P^{<\omega},\preceq, {}\conc)$ for some qo $P$.

\begin{Definition}\label{def:+-prop}
    ($\prsou$)
  We say that $(Q,\leq_Q,\cdot)$ is a \emph{multiplicative qo}$^+$ if it is a multiplicative
  qo with the following additional property: for every $a,b,c\in Q$,
  if $c\leq_Q ab$, then there are $a'\leq_Q a$, $b'\leq_Q b$ such that
  $c\equiv a'b'$.

  The definitions of multiplicative wfqo$^+$, multiplicative wqo$^+$, monoidal qo$^+$, monoidal wfqo$^+$ and monoidal wqo$^+$ are the obvious modification of the one above.
  In general, we call the \emph{$+$-property} the characteristic property of multiplicative qos$^+$'s.
\end{Definition}

% Notice that if $Q$ is monoidal qo$^+$, then every prime $p$ has the property
% that $p\leq_Q ab$ implies $p\leq_Q a$ or $p\leq_Q b$.
As anticipated, we include a result that is not really needed in the rest of the paper, but that can help familiarize with multiplicative qos$^+$, and gives some additional information about the interaction between the order and the product operation.

\begin{Lemma}\label{lem:neutral-el-wqo}
    ($\prsou$)
  If $Q$ is a set and $(Q,\leq_Q,\cdot)$ is a multiplicative wfqo$^+$, then it has a neutral element, i.e. it is a monoidal wfqo$^+$.
\end{Lemma}
\begin{proof}
    By well-foundedness, we can consider the set of minimal elements of $Q$.
  Suppose for a contradiction that there are two minimal elements $a,b\in Q$
  that are not equivalent. Then $a\leq_Q ab$, so there are $a'\leq_Q a$,
  $b'\leq_Q b$ with $a\equiv a'b'$. By minimality, it follows that $a\equiv a'$
  and $b\equiv b'$, hence by weak increasing-ness $b\leq_Q a$, contradiction. This
  proves that $Q$ has a minimal element, call it $m$.

  %Notice that if $m\equiv ab$, then $a\equiv b\equiv m$. Moreover, $m\leq_Q m^2$
  %implies that $m\equiv m^2$ by the $+$-property (and so $m\equiv m^n$ for every positive $n\in \omega$).

    We now show that $a\equiv am$ for all $a\in Q$. Suppose not, then by well-foundedness there is a $\leq_Q$-minimal $a$ such that $a\not\equiv am$, which entails $a<_Q am$. By the $+$-property, there is $a'\leq_Q a$ with $a\equiv a'm$, and by assumption on $a$ we can assume that $a'<_Q a$. But then we know that $a'm\equiv a'$, contradiction. 
    
    The fact that $a\equiv ma$ is proven similarly.
%    
%  Now suppose that $p\in Q$ is prime and $p\not\equiv m$. Since $p\leq_Q pm$, there are
  %$p'\leq_Q p$ and $m'\leq_Q m$ with $p\equiv p'm'$. Since $m'\equiv m$, by the fact
  %that $p$ is prime it follows that $p'\equiv p$, so we conclude that $pm\equiv m$.
  %Similarly we show that $mp\equiv p$.
%
%  Finally, given any $q\in Q$, take any decomposition of $q$ into primes: using the
 % previous results it is easy to see that $q\equiv qm \equiv mq$.
\end{proof}

We can now then limit ourselves to consider monoidal qos$^+$.

\begin{Lemma}\label{lem:qo+-prod}
    ($\prsou$)
  If $Q$ is a set and $(Q,\leq_Q,\cdot)$ is a monoidal qo$^+$, $p$ is prime and 
  $p\leq_Q a_0\cdot\dots\cdot a_n$, then there is $i\leq n$ with
  $p\leq_Q a_i$.
\end{Lemma}
\begin{proof}
  By induction on $n$. If $n=0$, the claim is trivial. Suppose it holds for
  $n$, we prove it for $n+1$. $p\leq_Q a_0\cdot\ldots\cdot a_n\cdot a_{n+1}$ implies
  $p\leq_Q (a_0\cdot\ldots \cdot a_n)a_{n+1}$, so by definition of multiplicative qo$^+$
  there are $a'\leq_Q a_0\cdot\ldots\cdot a_n$ and $b'\leq_Q a_{n+1}$ such
  that $p\equiv a'b'$. Since $p$ is prime, either $p\equiv a'$ or $p\equiv b'$
  holds. If the latter holds, then $p\leq_Q a_{n+1}$, as we wanted. If the former
  holds, then $p\leq_Q a_0\cdot\ldots\cdot a_n$ holds, and we conclude by induction.
\end{proof}

%Although this will not have a huge importance in the rest of the paper,
%we notice that monoidal wqo$^+$ always have a minimal element, which also happens
%to be the neutral element for $\cdot$.

We now move on to consider downward-closed subsets of $Q$.

\begin{Definition}
    ($\prsou$)
    Let $Q$ be a set and $(Q,\leq_Q)$ be a qo. We denote as $\dot{\mathcal{D}}(Q)$ the primitive recursive class of non-empty downward-closed (with respect to $\le_Q$) subsets of $Q$.  
    %We define the $\Delta_0$ formula $\varphi(x)$:
    %\[
    %\varphi(x) \Leftrightarrow (
    %   x\neq \emptyset \, \wedge \, \forall y\in x (y\in Q) 
    %   \, \wedge \, \forall z\in Q (\exists y\in x ( z\leq_Q y) \rightarrow
    %   z\in x)).
    %\]
    %We call $\dot{ \mathcal{D}}(Q)$ the class of sets that satisfy $\varphi$.
    %If $\varphi(x)$ holds, we say that $x\in \dot{ \mathcal{D}}(Q)$, and otherwise we write $x\not\in \dot{ \mathcal{D}}(Q)$. 

    Let $(Q,\leq_Q,\cdot)$ be a multiplicative qo.
    We define the primitive recursive class function $\bullet$ on $\dot{\mathcal{D}}(Q)$ as
      \[
        x\bullet y = \{ z\in Q: \exists x'\in x,y'\in y(z\leq_Q x'y')\}.
    \]
\end{Definition}
Note that, since $\prsou$ has no powerset axiom, we in general cannot assert that $\dot{\mathcal{D}}(Q)$ is a set. The class $\dot{\mathcal{D}}(Q)$ is primitive recursive, since if $Q$ is a set, the downward-closedness of a subset of $Q$ is $\Delta_0$-expressible.

%In essence, $\dot{ \mathcal{D}}(Q)$ is the class of non-empty downward-closed subsets of $Q$, and we equip it with the operation $\bullet$ such that, given sets $x$ and $y$, produces the downward closure of the set of products of elements of $x$ times some element of $y$. 
For the rest of this section, to enhance readability, we will refer to elements of $Q$ with lower-case letters, and to elements of $\dot{\mathcal{D}}(Q)$ with upper-case letters.

Again, we will shorten $A\bullet B$ to $AB$. 

The following result is not surprising.

\begin{Lemma}
        ($\prsou$)
  If $Q$ is a set and $(Q,\leq_Q,\cdot)$ is a monoidal qo, then so is $(\dot{\mathcal{D}}(Q),\subseteq, \bullet)$.
%  Under the previous assumptions, $(Q,\leq_Q,\cdot)$ is a monoidal wqo
%  if and only if $(D(Q),\subseteq,\bullet)$ is monoidal wfqo.
\end{Lemma}
\begin{proof}
  Associativity of $\bullet$ follows trivially from associativity of $\cdot$,
  and the same goes for monotonicity and weak increasing-ness. Notice that $E:=\{e\}$ is a downward-closed set, and $A= AE= EA$ for every downward-closed set $A$.
%  The fact that $(D(Q),\subseteq)$ is well-founded if and only if
%  $(Q,\leq_Q)$ is wqo is standard.
\end{proof}

In the following, given a qo $(Q,\leq_Q)$ and an element $q\in Q$, $q\uparrow$ denotes the set $\{p\in Q: q\leq_Q p\}$, and $q\downarrow$ denotes the set $\{p\in Q: p\leq_Q q\}$.

\begin{Lemma}
    ($\prsou(\mathsf{ch})$)
    If $Q$ is a set and $(Q,\leq_Q)$ is a wqo, then $\dot{\mathcal{D}}(Q)$ is a set and $(\dot{\mathcal{D}(Q)},\subseteq)$ is a well-founded qo.
\end{Lemma}
\begin{proof}
    This Lemma states a rather standard fact in the theory of wqos (see e.g.\ \cite[Section 4]{theory-relations-fraisse}): we give some details regarding the proof because this is the point where we start using the global choice function that $\prsou(\mathsf{ch})$ provides.

    Every downward-closed subset of $Q$ is the complement of an upward-closed subset. We claim that every upward-closed subset of $Q$ is of the form $\bigcup_{a\in A} a\uparrow$ for a finite subset $A$ of $Q$. If this was not the case, let $U$ be an upward-closed subset of $Q$ witnessing this: then, using the global choice function, we can use $U$ to produce an infinite bad sequence $(u_i)_{i\in\omega}$ in $Q$ by picking as $u_{i}$ an element of the non-empty set $U\setminus \bigcup_{j < i} u_j\uparrow$ (where clearly $U\setminus \bigcup_{j < 0} u_j\uparrow = U$). 

    %With the assumptions above, 
    Hence, $\dot{\mathcal{D}}(Q)$ can be seen as the set 
    \[
    \{ Q\setminus \bigcup_{a\in A} a\uparrow : A\in [Q]^{<\omega}\}\setminus \{\emptyset\}.
    \]
    %since in a wqo every upward-closed set only has finitely many minimal elements (a proof of this fact can be found in any standard reference on wqo's, like \cite[Section 4]{theory-relations-fraisse}, and it is easily seen that it goes through in a weak base theory), and every downward closed set is the complement of an upward-closed one.

    %The fact that $(\dot{ \mathcal{D}}(Q),\subseteq)$ is well-founded is standard.
    %It is a standard fact that, if the set $Q$ is a wqo, then $\dot{ \mathcal{D}}(Q)$ is actually a set, and the ordering $\subseteq$ is well-founded (again, see for instance \cite[Section 4.4.1]{theory-relations-fraisse}, or any introduction to wqo theory). %this is essentially the same argument as in the implication $1\Rightarrow 2$ in \cite[Corollary 5.14]{nwfr-pakhomov-solda} with $\alpha=1$. \todo{why enu there?}
    
    In a similar fashion, one can extract an infinite bad sequence in $Q$ from an infinite descending sequence in $\dot{\mathcal{D}}(Q)$.
\end{proof}

With this, we can proceed in the study of $(\dot{\mathcal{D}}(Q),\subseteq,\bullet)$ as a monoidal qo. 

\begin{Lemma}\label{lem:prod-dec-dc}
($\prsou(\mathsf{ch})$)
  Let $Q$ be a set and let $(Q,\leq_Q,\cdot)$ be a monoidal wqo$^+$. Then for every
  $A,B,C\in \dot{\mathcal{D}}(Q)$, if $C\subseteq AB$, there is a number $n\in\omega$
  and sets $A_0,\dots, A_{n-1}$, $B_0,\dots,B_{n-1}$ in $\dot{\mathcal{D}}(Q)$ such that
  for every $i<n$ we have $A_i\subseteq A$ and $B_i\subseteq B$, and
  $C=\bigcup_{i<n} A_iB_i$.
\end{Lemma}
\begin{proof}
  For every $a\in A$, let us consider $B_a=\{ b\in B: ab\in C\}$. We claim that there are only finitely many distinct values of $B_a$. Notice that
  every $B_a$ is downward-closed. Moreover, if $a\leq_Q a'$, then
  $B_{a'}\subseteq B_a$. Therefore the relation $a\preceq a'\iff B_{a'}\subseteq B_a$ on $Q$ is both a wqo and conversely well-founded. Observe that it entails that $\preceq$ has finitely many equivalence classes. Indeed, by the use of the choice function we construct a possibly finite sequence $a_0,a_1,\ldots$ such that for each $n$ either $a_{n-1}$ is the last element of the sequence or $a_n$ is not in the same $\preceq$ class as any of $a_0,\ldots,a_{n-1}$. If the sequence construction process has terminated at a finite stage, then we are done. Otherwise, by the use of Ramsey's Theorem for pairs, we see that the sequence of $a_i$'s has an infinite subsequence that either forms an infinite antichain or an infinite descending chain or an infinite ascending chain. The first and the second options are not possible since $\preceq$ is a wqo and the third option is not possible since $\preceq$ is conversely well-founded.
  
  We pick $a_0,\ldots,a_{n-1}$ such that $\{B_{a_i}\mid i<n\}$ contains all possible values of $B_{a}$. We put $B_i=B_{a_i}$. Let $A_{i}=\{a\in A \mid \forall b\in B_{i} (ab\in C)\}$. We claim that the $A_i$'s and $B_i$'s that we constructed satisfy the claim of the lemma. Indeed given $c\in C$, there are $a\in A$ and $b\in B$ such that $c\le_Q ab$. Since $Q$ is wqo$^+$, there are $a'\le_Q a$ and $b'\le_Q b$ such that $c\equiv a'b'$. To finish the proof it will be enough to find $i$ such that $a'\in A_i$ and $b'\in B_i$.
  
  We pick the unique $i$ such that $B_i=B_{a_i}=B_{a'}$. To see that $a'\in A_i$ we have to verify that $\forall b\in B_{i} (a'b\in C)$, which holds since $B_i=B_{a'}$ and $B_{a'}$ is precisely the set of $b\in B$ such that $a'b\in C$.  And to see that $b'\in B_i=B_{a'}$ we have to verify that $a'b'\in C$, which of course holds since $a'b'\equiv c\in C$.\end{proof}

We now move to consider the space of ideals of $Q$.

\begin{Definition}
    ($\prsou$)
    Let $Q$ be a set and $(Q,\leq_Q)$ be a qo. Then $\dot{\mathcal{I}}(Q)$ is the subclass of $\dot{\mathcal{D}}(Q)$ that only contains upward-directed sets, i.e.
    \[
        A\in \dot{\mathcal{I}}(Q) \Leftrightarrow A\in \dot{\mathcal{D}}(Q) \, \wedge \,
        \forall a,b\in A \exists c\in A (a\leq_Q c\wedge b\leq_Q c).
    \]
\end{Definition}

It is clear that the class $\dot{\mathcal{I}}(Q)$ is primitive recursive. Elements of $\dot {\mathcal{I}}(Q)$ are called ideals of $Q$.

\begin{Lemma}\label{lem:id-mult-qo}
    ($\prsou(\mathsf{ch})$)
  If $Q$ is a set and $(Q,\leq_Q,\cdot)$ is a monoidal qo, then so is $(\dot{\mathcal{I}}(Q),\subseteq, \bullet)$.
  Under the previous assumptions, if $(Q,\leq_Q)$ is a wqo, then
   $(\dot{\mathcal{I}}(Q),\subseteq,\bullet)$ is a monoidal wfqo.
\end{Lemma}

%We again denote $I\bullet J$ simply by $IJ$.

\begin{proof}
  We verify that if $I,J\in \dot{\mathcal{I}}(Q)$, then $IJ$ is still in $\dot{\mathcal{I}}(Q)$, all the other
  properties are verified as in the case of $\dot{\mathcal{D}}(Q)$. Let $c_0,c_1\in IJ$, then there are
  $a_0,a_1\in I$ and $b_0,b_1\in J$ such that $c_0\leq_Q a_0b_0$ and
  $c_1\leq_Qa_1b_1$. Let $a_2$ be such that $a_0,a_1\leq_Q a_2$, and $b_2$ be such that $b_0,b_1\leq_Q b_2$.
  Then $a_2b_2\in IJ$ and $c_0,c_1\leq_Q a_2b_2$.
\end{proof}

We recall some known properties of ideals in the case that $Q$ is a wqo.

\begin{Lemma}\label{lem:known-prop}
($\prsou(\mathsf{ch})$)
  Let $Q$ be a set and $(Q,\leq_Q)$ a wqo. Then
  \begin{itemize}
  \item if $D\in \dot{\mathcal{D}}(Q)$, then there are $n\in\omega$ and $I_0,\dots,I_{n-1}\in \dot{\mathcal{I}}(Q)$
    such that $D= \bigcup_{i<n} I_i$.
  \item If $J\subseteq\bigcup_{i<n}I_i$, for $J,I_0,\dots, I_{n-1}\in \dot{\mathcal{I}}(Q)$, then there
    is $k<n$ such that $J\subseteq I_k$.
    Moreover, if $J=\bigcup_{i<n}I_i$, for $J,I_0,\dots, I_{n-1}\in \dot{\mathcal{I}}(Q)$, then there
    is $k<n$ such that $J= I_k$.
  \item The product distributes over the union: for every $I,J,K,H\in \dot{\mathcal{I}}(Q)$,
    $(I\cup J)(K\cup H)= IK \cup IH \cup JK \cup JH$.
  \end{itemize}
\end{Lemma}
\begin{proof}
  We prove the first claim first.
  We know that $(\dot{\mathcal{D}}(Q),\subseteq)$ is well-founded if $Q$ is a wqo. Hence, we can
  proceed by $\Delta_0$-induction on the order $\subseteq$. Suppose for a contradiction
  that $D$ is $\subseteq$-minimal such that it is not the union of finitely many
  ideals. In particular, $D$ is itself not an ideal, so there are $a,b\in D$
  such that there is no $c\in D$ with $a,b\leq_Q c$. Let $D_a=\{ d\in D: a\not\leq_Q d\}$,
  and $D_b=\{d\in D: b\not\leq_Q d\}$. Then $D\subseteq D_a\cup D_b$,
  and $D_a,D_b\subset D$. Hence, $D_a$ and $D_b$ are finite unions of ideals, and
  then so is $D$.

  Now for the second statement. Suppose for a contradiction that $J$ was
  not contained in any
  of the ideals $I_i$, then for every $i<n$ there is an element
  $j_{i}$ such that $j_i\in J\setminus I_i$. Then there is $j\in J$ above all the
  $j_i$'s, and this $j$ is in $I_i$ for some $i<n$, contradiction.
  The ``Moreover'' part is proved in the same way.

  The proof of the last point is routine. Notice that this result holds already for downward-closed sets.
\end{proof}

We are finally able to show that, if $(Q,\leq_Q,\cdot)$ has the $+$-property, then so does $(\dot{\mathcal{I}}(Q),\subseteq,\bullet)$.

\begin{Lemma}\label{lem:wqo-plus-to-id-plus}
($\prsou(\mathsf{ch})$)
  If $Q$ is a set and $(Q,\leq_Q,\cdot)$ is a monoidal wqo$^+$, then $(\dot{\mathcal{I}}(Q),\subseteq,\bullet)$
  is a monoidal wfqo$^+$.
\end{Lemma}
\begin{proof}
  The only thing that needs to be shown is that if $J\subseteq IH$, then
  there are ideals $I'\subseteq I$ and $H'\subseteq H$ such that $J=I'H'$.
  Given $J\subseteq IH$, by \Cref{lem:prod-dec-dc} there are downward-closed
  sets $A_0,\dots,A_{n-1}$ and
  $B_0,\dots, B_{n-1}$ such that $J=\bigcup_{i<n} A_iB_i$. By \Cref{lem:known-prop},
  every $A_i$ is the union of finitely many ideals, $A_i=I_{i,0}\cup\ldots \cup I_{i,m_i}$, and the same goes for the $B_i$ being represented as the unions of ideals $B_i=H_{i,0}\cup\ldots \cup H_{i,k_i}$.
  Again by \Cref{lem:known-prop}, since the product distributes over the union,
  $J$ is a union of products of ideals \[J=\bigcup_{i<n}\bigcup_{j<m_i}\bigcup_{l<k_i} I_{i,j}H_{i,l}.\]
  Invoking \Cref{lem:known-prop} again, we conclude that for some particular $i,j,l$ we have $J=I_{i,j}H_{i,l}$.
\end{proof}

%We say that an ideal $J$ is \emph{prime} if it is a prime element of
%$(I(Q),\subseteq,\bullet)$, i.e.\ if $J= IH$, then $J=I$ or $J=K$.

We can now prove the main result of this section, namely that, speaking loosely, the ideals in $\dot{\mathcal{I}}(Q)$ that happen to be primes in the monoid $(\dot{\mathcal{I}}(Q),\bullet)$ are completely determined by the primes of $Q$, provided the latter is a monoidal wqo$^+$. To do this, we show that the elements of $\prim(\dot{\mathcal{I}}(Q))$ only come in two forms: this is a generalization of 
\cite[Theorem 4.11]{ideal-decomp-goubault-larrecq-halfon-karakndikar-narayan-schnoebelen}, where only the case where $Q$ is the Higman ordering on a wqo $P$ was considered. We remark that, as far as we can tell, our proof is rather different from the one given in \cite{ideal-decomp-goubault-larrecq-halfon-karakndikar-narayan-schnoebelen}.

For the proof, we need some additional terminology.

\begin{Definition}
($\prsou$)
  Let $Q$ be a set, let $(Q,\leq_Q,\cdot)$ be a multiplicative qo and $A\subseteq Q$. By $A^*$ we denote the
  set of finite products of $A$, namely $A^*= \{ a\in Q:
  \exists \{a_0,\dots,a_{n-1}\}\in [A]^{<\omega}
  (a\equiv a_0\cdot\dots \cdot a_{n-1})\}$.
\end{Definition}

%We now show that the prime ideals of $Q$ can only have two shapes.
Although this is not really needed in what follows, we observe that, if $Q$ is a monoidal wfqo$^+$ and $D\in \dot{\mathcal{D}}(Q)$, then $(\prim(Q)\cap D)^*$ is always an ideal. Indeed, being closed under products, $(\prim(Q)\cap D)^*$ is clearly upward-directed, so we only have to show that it is downward-closed. Consider $a\in (\prim(Q)\cap D)^*$ and $b\in Q$ with $b\leq_Q a$. If $b$ is prime, then by \Cref{lem:qo+-prod} it is below a prime factor of $a$, and so $b\in D$ since $D$ is downward-closed. Similarly, if $b$ is not prime, then every one of its prime factors is in $D$, and we conclude that $b\in (\prim(Q)\cap D)^*$ since $(\prim(Q)\cap D)^*$ is closed under finite products.

\begin{Theorem}\label{theo:id-two-forms}
($\prsou(\mathsf{ch})$)
  Let $Q$ be a set, let $(Q,\leq_Q,\cdot)$ be a monoidal wqo$^+$, and let $I\in \dot{\mathcal{I}}(Q)$ be prime. Then,
  either $I= (\prim(Q) \cap I)^*$, or $I=(\prim(Q)\cap I)\downarrow$.
\end{Theorem}
\begin{proof}
  Suppose that $I$ is not of the form $(\prim(Q)\cap I)^*$. Then, there is a
  sequence $q_0,\dots q_n$ of primes in $I$ such that
  $q_0\cdot\dots\cdot q_n\not\in I$. Notice that $n\geq 1$. We can suppose that
  $n$ is such that $q_0\cdot\dots\cdot q_{n-1} \in I$, by removing tails of the
  sequence. For brevity, we denote $q_0\cdot\dots\cdot q_{n-1}$ by $a$.

  Let $D_a=\{ c\in I: a\not\leq_Q c\}$ and $D_q=\{c\in I: q_n\not\leq_Q c\}$,
  they are both downward-closed sets. We claim that
  \[
    I\subseteq D_a\bullet (\prim(Q)\cap I)\downarrow \bullet D_q.
  \]
  We prove this claim by considering any $c\in I$ and showing that $c$ is also inside the product in the right-hand side of the inclusion above. Of course, $c$ is equivalent to the product of finitely many primes,
  $c\equiv p_0\cdot\dots\cdot p_m$.
  % If $m=0$, then $c$ is prime itself, so
  % indeed $c\in D_a\bullet (I\cap Pr(Q))\downarrow \bullet D_q$. So we can suppose
  % that $m\geq 1$.
  
  Suppose that for no $i\leq m$ we have that $q_n\leq_Q p_i$. Then, by
  \Cref{lem:qo+-prod}, $q_n\not\leq_Q p_0\cdot\dots\cdot p_m$, so
  $c\in D_q$, as we wanted. So we can suppose that
  there is such $i$ with $q_n\leq_Q p_i$: let
  $j\leq m$ be the largest of these $i$'s, so that we have $c\equiv p_0\cdot
  \ldots\cdot p_{j-1} \cdot p_j \cdot p_{j+1}\cdot\ldots\cdot p_m$.

  By definition of $j$ and \Cref{lem:qo+-prod},
  $q_n\not\leq_Q p_{j+1}\cdot\ldots\cdot p_m$.
  Moreover, $p_0\cdot\ldots\cdot p_{j-1}\not\geq_Q a$, since otherwise  $aq_n\leq_Q
  p_0\cdot\ldots\cdot p_{j-1} p_j\leq_Q c$, which would entail that $aq_n\in I$, contradicting the definition of the sequence $q_0,\ldots q_n$.
  It follows that $p_0\cdot\ldots\cdot p_{j-1}\in D_a$: since we also have that $p_j\in (\prim(Q)\cap I)\downarrow$ and $p_{j+1}\cdot\ldots\cdot p_m\in D_q$, the claim is proved.

  Since we know that $D_a$, $D_q$ and $(\prim(Q)\cap I)\downarrow$ are downward-closed
  sets, we can apply \Cref{lem:known-prop} to first express each of the
  three sets above as a finite union of ideals, then to distribute the product over
  the union, and finally to see that $I$ is contained in one of the elements
  of the union. Summarizing, we have that $I\subseteq I_0I_1I_2$, where
  $I_0\subseteq D_a$, $I_1\subseteq (\prim(Q)\cap I)\downarrow$ and
  $I_2\subseteq D_q$ are ideals. Since $I$ is prime, by \Cref{lem:qo+-prod}
  $I\subseteq I_0$ or $I\subseteq I_1I_2$. Since $a\in I$, it cannot be that
  $I\subseteq I_0$. Applying the same argument to $I_1I_2$, we conclude that
  $I\subseteq (\prim(Q)\cap I)\downarrow$. On the other hand, it is always the
  case that $(\prim(Q)\cap I)\downarrow\subseteq I$, so we can conclude that
  $I=(\prim(Q)\cap I)\downarrow$, as we wanted. 
\end{proof}

\section{$\dot V^*(Q)$ and $\dot I^*(Q)$ as multiplicative qos}\label{sec:vq-iq-mult-qo}

In this section, we assume that the language $L$ of the structure of urelements is given by a binary relation $\leq_Q$ and a binary function symbol $\cdot$, and we assume that the structure $(Q,\leq_Q,\cdot)$ of urelements is a monoidal qo. With this assumption, we show how to extend the $\cdot$ operation to $\dot V^*(Q)$ and $\dot I^*(Q)$ in such a way that $(\dot V^*(Q), \lesssim_Q^*,\cdot)$ and $(\dot I^*(Q),\lesssim_Q^*,\cdot)$ are multiplicative qos.

\begin{Definition}
    ($\prsou$)
    %Let $(Q,\leq_Q,\cdot)$ be a monoidal qo. 
    We extend the operation $\cdot$ to $\dot V^*(Q)$ by recursion as follows:
    \[
        x\cdot y =
        \begin{cases}
            x\cdot y &\text{ if } x,y\in Q \\
            \{x\}\cdot y &\text{ if } x\in Q,y\not\in Q \\
            x\cdot \{ y \} &\text{ if } x\not\in Q, y\in Q \\
            \{ x'\cdot y': x'\in x, y'\in y\} &\text{ if }x,y\not\in Q
        \end{cases}
    \]
    We will abuse notation and denote $x\cdot y$ by $xy$.
\end{Definition}

We just have to verify that associativity, monotonicity and weak increasing-ness hold. We do it for $\dot V^*(Q)$ in the following three Lemmas.

\begin{Lemma}
    ($\prsou$)
    For all $x_0,x_1,x_2\in \dot V^*(Q)$, $(x_0x_1)x_2\sim^*_Q x_0(x_1x_2)$.
\end{Lemma}
\begin{proof}
    This is a straightforward verification. %\todo{add something}
\end{proof}

\begin{Lemma}
    ($\prsou$)
    For every $x,y\in \dot V^*(Q)$, $x\lesssim_Q^* xy$.
\end{Lemma}
\begin{proof}
    We proceed by induction on $\max\{\rank x,\rank y\}$.
    \begin{itemize}
        \item If $x,y\in Q$, then the claim follows from $Q$ being a monoidal qo.
        \item If $x\in Q$, $y\not\in Q$, $xy=\{xy':y'\in y\}$, and $x\lesssim_Q^* xy'$ by inductive assumption.
        \item If $x\not\in Q$, $y\in Q$, $xy=\{x'y:x'\in x\}$, and from $x'\lesssim_Q^* x'y$ we conclude $x\lesssim_Q^* xy$.
        \item If $x,y\not\in Q$, $xy=\{x'y':x'\in x,y'\in y\}$. Fixing $y''\in y$, we have $x'\lesssim_Q^* x'y''$ for every $x'\in x$ (by inductive assumption), and so $x\lesssim_Q^* x\cdot \{y''\}\lesssim_Q^* xy$. 
    \end{itemize}
\end{proof}

\begin{Lemma}
    ($\prsou$)
    If $x_0\lesssim^*_Q x_1$ and $y_0\lesssim^*_Q y_1$, then $x_0y_0\lesssim^*_Q x_1y_1$.
\end{Lemma}
\begin{proof}
    We prove the claim by induction on $\alpha=\max\{ \rank x_0, \rank x_1,\rank y_0,\rank y_1\}$, where we use the convention that the rank of urelements is $-1$. If $\alpha=-1$, then the claim follows from the fact that $Q$ is a monoidal qo.

    Suppose then that $\alpha>-1$, and assume that the maximum is achieved by $\rank x_0$, the other cases being similar. 

    We have to consider several cases. First, suppose that $y_0\in Q$, so $x_0y_0=\{x_0'y_0:x_0'\in x_0\}$.
    \begin{itemize}
        \item If $x_1,y_1\in Q$, then $x_0\lesssim_Q^* x_1$ implies $x_0'\lesssim_Q^* x_1$ for all $x_0'\in x_0$, hence by induction we have $x_0'y_0\lesssim_Q^* x_1y_1$, and so $x_0y_0\lesssim_Q^* x_1y_1$.
        \item If $x_1\not\in Q$ and $y_1\in Q$, then $x_1y_1=\{x_1'y_1:x_1'\in x_1\}$, and it is easy to conclude using the inductive hypothesis.
        \item If $x_1\in Q$ and $y_1\not\in Q$, then $x_1y_1=\{x_1y_1':y_1'\in y_1\}$. By our assumptions it follows that $x_0'\lesssim_Q^* x_1$ for all $x_0'\in x_0$, and $y_0\lesssim_Q^* y_1''$ for some $y_1''\in y_1$, hence $x_0y_0\lesssim_Q^* \{ x_1y_1'' \}$ follows by inductive assumption.
        \item If $x_1,y_1\not\in Q$, then $x_1y_1=\{ x_1'y_1':x_1'\in x_1,y_1'\in y_1\}$. We know that for every $x_0'\in x_0$ there is $x_1''\in x_1$ such that $x_0'\lesssim_Q^* x_1''$, and there exists $y''_1\in y_1$ with $y_0\lesssim_Q^* y_1''$. By inductive assumption, it follows that $x_0y_0\lesssim_Q^*\{x_1'y_1'':x_1'\in x_1\}$, and the claim follows.
    \end{itemize}
    The case of $y_0\not\in Q$ is similar.
\end{proof}

Finally, it is easy to see that $(\dot I^*(Q),\lesssim_Q^*,\cdot)$ is a multiplicative qo as well.

\begin{Lemma}
    ($\prsou$)
    For every $\alpha$ and $x,y\in \dot I^*_\alpha(Q)$, we have that $xy\in \dot I^*_\alpha(Q)$.
\end{Lemma}
\begin{proof}
    We prove the claim by induction on $\max\{\rank x,\rank y\}$. It is immediate to see that if $x'\in x$ and $y'\in y$, then $x'y'\in \dot I^*_\beta(Q)$ for some $\beta<\alpha$. To see that $xy$ is upward-directed, notice that for every $x'y'$ and $x''y''$ in $xy$, we can find $x'''y'''$ such that $x',x''\lesssim^*_Q x'''$ and $y',y''\lesssim^*_Q y'''$, thus getting that $x'y',x''y''\lesssim^*_Q x'''y'''$.
\end{proof}

As we will see, in the next section we will need to work with sets, rather than just classes. We then proceed to ``approximate'' the classes $\dot I^*_\alpha(Q)$ by modifying the way they are generated. We remark that what we do here is essentially analogous to the content of \cite[Section 4.2]{nwfr-pakhomov-solda}.

\begin{Definition}
    ($\prsou$) Suppose that $Q$ is a set. We define the primitive set-recursive 
    function $\dot{\mathcal{I}}^f$ on qos $(S,\leq_S)$ as follows: 
    \[
        \dot{\mathcal{I}}^f(S,\leq_S)= \{ T\subseteq S: \exists A\in [S]^{<\omega} (T=S\setminus \bigcup_{a\in A} a\uparrow_{\leq_S}) \text{ and $T$ is $\leq_S$-upward-directed}\}\setminus \{\emptyset\}
    \]
    We define the sets $\hat{I}^*_\alpha (Q)$ by primitive set-recursion by putting 
     \[
        \hat{I}^*_\alpha (Q) = Q \cup 
        \bigcup_{\beta<\alpha} 
        \dot{\mathcal{I}}^f(\hat{I}^*_\beta (Q),\lesssim^*_Q).
     \]
     We denote by $\hat{I}^*(Q)$ the class constituted by the union of the $\hat{I}^*_\alpha(Q)$'s, ordered by $\lesssim^*_Q$.
\end{Definition}

Notice that, as long as $Q$ is a set, all the $\hat{I}^*_\alpha(Q)$ are sets as well.

We then proceed to extend the operations of multiplication and finding primes to $\hat{I}^*_\alpha(Q)$. Due to the restricted form the elements of $\hat{I}^*_\alpha(Q)$ can have, the multiplication will not be total in this setting. 

\begin{Definition}
    ($\prsou$) Let $Q$ be a set and $(Q,\leq_Q,\cdot)$ be a monoidal qo and $\alpha$ an ordinal. For every $x,y\in \hat{I}^*_\alpha(Q)$, we set
    \[
        x\, \hat{\cdot} \, y =
        \begin{cases}
            z & \text{ if there exists $z\in \hat{I}^*_\alpha(Q)$ with $z\sim^*_Q x\cdot y$ and} \\
            & \text{ for no $\alpha'<\alpha$ there exists $z'\in \hat{I}^*_{\alpha'}(Q)$ with $z'\sim^*_Q x\cdot y$}\\
            \emptyset & \text{ otherwise}            
        \end{cases}
    \]
    Given a subset $S$ of $\hat{I}^*_\alpha(Q)$, we denote by $S^*$ the set of finite $\hat{\cdot}\,$-products of elements of $S$.
\end{Definition}

    %We then define the map $\widehat{\prim}$ on subset $S$ of $\hat I^*_\alpha(Q)$ as
\begin{Remark}
    Further in the paper we consider $\hat \cdot$ to be the ``default'' multiplication operation on $\hat{I}^*_\alpha(Q)$ and we identify $\hat{I}^*_\alpha(Q)$ with the structure $(\hat{I}^*_\alpha(Q),\lesssim^*_Q,\hat\cdot)$. 
    
    In general we cannot guarantee that the $(\hat{I}^*_\alpha(Q),\lesssim^*_Q,\hat\cdot)$ will be monoidal qos. However, as we will see later, under suitable $\alpha$-wqo-ness assumptions they will be monoidal qos. 
    
    For technical reasons it will sometimes be convenient to formally consider sets $\prim(\hat{I}^*_\alpha(Q))$ even when $\hat{I}^*_\alpha(Q)$ are not monoidal qos. For this we formally drop the requirements on the structure to be a monoidal qo in the definition of the set of primes when talking about $\hat{I}^*_\alpha(Q)$. In other words, for every ordinal $\alpha$, we define
    \[
        \prim(\hat{I}^*_\alpha(Q)) = \{ x\in \hat{I}^*_\alpha(Q): \exists y\in \hat{I}^*_\alpha(Q)( x\hat \cdot y \nsim_Q^* y)\land \forall y,z\in \hat{I}^*_\alpha(Q)(x\sim^*_Q y\, \hat{\cdot} \, z \rightarrow x\sim^*_Q y \text{ or } x\sim^*_Q z)\}.
    \]     
\end{Remark}
%\todo{just extend the previous version and extend it to this case. Remove the hats afterwards}
Again, notice that, with the assumption that $Q$ is a set, $\hat{\cdot}$ is a primitive set-recursive function, and that ${\prim}(\hat{I}^*_\alpha(Q))$ is a set for every ordinal $\alpha$.

\begin{Lemma}\label{lem:wqo-set}
    ($\prsou(\mathsf{ch})$) Let $(Q,\leq_Q)$ be a qo with $Q$ a set and $\alpha$ an ordinal. If for every $\beta<\alpha$ $\hat{I}^*_\beta(Q)$ is a wqo, then for every $x\in \dot I^*_\alpha(Q)$ there exists $y\in \hat{I}^*_\alpha(Q)$ such that $y\sim_Q x$.

    In particular, if $\hat{I}^*_\alpha(Q)$ is a wqo, then so is $\dot I^*_\alpha(Q)$.
    %\todo{note: we cannot conclude that hat is wf if the ordinal is limit, it is the "usual" problem of the star setting from the last paper}

    Moreover, supposing additionally that $(Q,\leq_Q,\cdot)$ is a monoidal qo, we have that for every $x,y\in \hat{I}^*_\alpha(Q)$, $xy\sim^*_Q  x\, \hat{\cdot} \, y \in \hat{I}^*_\alpha(Q)$, and so $(\hat{I}^*_\alpha(Q),\lesssim^*_Q,\hat{\cdot})$ is a monoidal qo. %Finally, $x\in \widehat{\prim}(\hat{I}^*_\alpha(Q))$ if and only if $x\in \prim(\hat{I}^*_\alpha(Q))$.
\end{Lemma}
\begin{proof}
    We point out that this result is analogous to \cite[Lemma 4.11]{nwfr-pakhomov-solda}. 
    
    We prove the claim by $\in$-induction. If $x\in Q$, then the claim holds by considering $x=y$. So suppose that $x\in \dot I^*_\alpha(Q)\setminus Q$. Suppose that $x\in \dot I^*_{\gamma+1}(Q)\setminus \dot I^*_\gamma(Q)$.  We can then consider the set
    \[
        y=\{ y'\in \hat{I}^*_\gamma(Q): \exists x'\in x (y'\lesssim^*_Q x')\}.
    \]
    First, we claim that $y\sim^*_Q x$. By definition, we have that $y\lesssim^*_Q x$, so we only need to show $x\lesssim^*_Q y$, i.e.\ that $\forall x'\in x\exists y'\in y (x'\lesssim^*_Q y')$. By inductive assumption, for every $x'\in \dot I^*_\gamma(Q)$ there is $y'\in \hat{I}^*_\gamma(Q)$ with $y'\sim^*_Q x'$: since every one of such $y'$ is an element of $y$, the claim follows.

    Next, we claim that $y\in \hat{I}^*_\alpha(Q)$. Since $x\sim_Q^* y$, the fact that $y$ is upward-directed follows from the fact that $x$ is. Moreover, by definition, $y$ is a downward-closed subset of $\hat{I}^*_\gamma(Q)$ and since we are assuming that $\hat{I}^*_\gamma(Q)$ is a wqo, using the function $\mathsf{ch}$ we can find finitely many generators of the complement of $\hat{I}^*_\gamma(Q)$. Thus $y\in \hat{I}^*_\gamma(Q)$.

    For the ``moreover'' part, it suffices to notice that, by what we said above, there is in $\hat{I}^*_\alpha(Q)$ a $z$ equivalent to $xy$, so the operation $\hat{\cdot}$ is total with our assumptions: the verification that it enjoys the properties of a monoidal qo follows then from the same result for $\cdot$. %The consideration in the two prime functions follows easily from the previous results. %\todo{explain more?}
\end{proof}

We will need another result about the sets $\hat{I}^*_\alpha(Q)$, which gives the connection with what we did in \Cref{sec:id-combinatorics}.

\begin{Lemma}\label{lem:next-lev-mult-qo}
    ($\prsou(\mathsf{ch})$) Let $(Q,\leq_Q,\cdot)$ be a monoidal qo such that $Q$ is a set, and let $\alpha$ be an ordinal. Suppose moreover that $\hat{I}^*_\beta(Q)$ is a wqo for every $\beta\leq\alpha$. 
    Then $(\dot{\mathcal{I}}(\hat{I}^*_\alpha(Q)),\subseteq)$ is the restriction of $(\hat{I}^*_{\alpha+1}(Q),\lesssim^*_Q)$ to $\dot{\mathcal{I}}(\hat{I}^*_\alpha(Q))$ and this restriction contains exactly one representative for each $\sim_Q^*$ equivalence class.
    
    %Then each $\sim^*_Q$ equivalence class in $\hat{I}^*_{\alpha+1}(Q)$ has exactly one representative from $\dot{ \mathcal{I}}(\hat{I}^*_\alpha(Q))$ and $\lesssim_Q^*$ relation restricted to $\dot{ \mathcal{I}}(\hat{I}^*_\alpha(Q))\subseteq \hat{I}^*_{\alpha+1}(Q)$ is exactly $\subseteq$.
    
    % Then, the identity map is an order isomorphism $i_{\alpha+1}: (\dot{ \mathcal{I}}(\hat{I}^*_\alpha(Q)),\subseteq) \to (\hat{I}^*_{\alpha+1}(Q)/\sim^*_Q,\lesssim^*_Q)$.

    Moreover, under the previous assumptions, 
    %$x \, \hat{\cdot} \, y \sim^*_Q x \cdot y$ for every $x,y\in \hat{I}_{\alpha}(Q)$, and 
    for every $x,y\in \dot{\mathcal{I}}(\hat{I}^*_\alpha(Q))$, $x \, \hat{\cdot} \, y \sim^*_Q x\bullet y \sim^*_Q x\cdot y$. 
    %Finally, $\prim(x)\sim^*_Q \widehat{\prim}(x)$.
\end{Lemma}
\begin{proof}
%\todo{change this according to the new formulation}
    We first observe that, since $\hat{I}^*_\alpha(Q)$ is a wqo, then the operations $\dot{\mathcal{I}}$ and $\dot{\mathcal{I}}^f$ coincide, again since every downward-closed set of a wqo is the complement of a finitely generated upward-closed set. In particular, $\dot{\mathcal{I}}(\hat{I}^*_\alpha(Q))\subseteq \hat{I}^*_{\alpha+1}(Q)$.% We observe that this implies that $\dot I(\hat{I}_\alpha(Q))$ is a set, by \Cref{lem:wqo-set}.

    Next, we want to see that, for any $x,y\in \dot{\mathcal{I}}^f(\hat{I}^*_\alpha(Q))$, $x\subseteq y$ holds if and only if $x \lesssim^*_Q y$. The left-to-right implication is immediate from the definition, so we can focus on the other direction. Suppose that $x,y\in \dot{\mathcal{I}}(\hat{I}^*_\alpha(Q))$ are such that $x\lesssim^*_Q y$: this means that $\forall x'\in x\exists y'\in y( x'\lesssim^*_Q y')$. But every ideal is downward-closed, so this implies that $\forall x'\in x(x'\in y)$, as we wanted.

    %We let $i_{\alpha+1}$ be the identity map. 
    Now, we want to show that every element of $\hat{I}^*_{\alpha+1}(Q)$ is $\sim^*_Q$-equivalent to an element of $\dot{\mathcal{I}}^f(\hat{I}^*_\alpha(Q))$. 
    Since $\hat{I}^*_{\alpha+1}(Q)= \hat{I}^*_\alpha(Q) \cup \dot{\mathcal{I}}^f(\hat{I}^*_\alpha(Q))=\hat{I}^*_\alpha(Q) \cup \dot{\mathcal{I}}(\hat{I}^*_\alpha(Q))$, we just need to prove that every element of $\hat{I}^*_\alpha(Q)$ is $\sim^*_Q$-equivalent to an element of $\dot{\mathcal{I}}(\hat{I}^*_\alpha(Q))$. So let us consider $x\in \hat{I}^*_\alpha(Q)$: if $x$ is a set, then it simply belongs to $\dot{\mathcal{I}}(\hat{I}^*_\alpha(Q))$. If instead $x\in Q$, notice that $x\sim^*_Q \{y\in Q:y\leq_Q x\}\in \dot{\mathcal{I}}(\hat{I}^*_\alpha(Q))$, as we wanted.

    %it is easily seen that $i_{\alpha+1}$ is order-preserving: if $x\subseteq y$, then $x\lesssim^*_Q y$. So we just need to prove that $i_{\alpha+1}$ is order-reflecting. 

    The ``moreover'' part follows easily. %\todo{explain more?}
    %For the ``moreover'' part, by \Cref{lem:wqo-set}, for every $z\in \dot I^*_\alpha(Q)$ there is $z'\in \hat{I}_\alpha(Q)$ with $z'\sim^*_Q z$: therefore, there is $z\in \hat{I}_\alpha(Q)$ such that $z\sim^*_Q x\cdot y$ for every $x,y\in \hat{I}_\alpha(Q)$. The other case works similarly.
\end{proof}

We conclude this section by showing that \Cref{theo:id-two-forms} can be somewhat improved if we assume that the primes of $Q$ form a downward-closed set.

\begin{Definition}
($\prsou$)
We say that $x\in \hat{I}^*_\alpha(Q)$ is \emph{multiplicatively closed} if it is a set and for all $y,z\in x$, there is $w \in x$ such that $y\cdot z \lesssim_Q^* w$.    
\end{Definition}

We notice that if $x\in \hat{I}^*_\alpha(Q)$ is multiplicatively closed, then it is also idempotent, i.e.\ $x\cdot x \sim^*_Q x$.

\begin{Theorem}\label{cor:id-mult-or-itprim}
($\prsou(\mathsf{ch})$)
    Suppose that $Q$ is a monoidal qo with the property that for every $p,q\in Q$, if $p\leq_Q q$ and $q\in\prim(Q)$, then $p$ is either a neutral element or a prime.
    Suppose that $\hat{I}^*_\alpha(Q)$ is a monoidal wqo$^+$, 
    and let $x\in {\prim}(\hat{I}^*_{\alpha+1}(Q))$. Then at least one of the following two options holds:
    \begin{itemize}
        \item $x\sim^*_Q({\prim}(\hat{I}^*_\delta(Q))\cap x)^*$ for some $\delta\leq \alpha$, or
%        \item for every $y\in x$ there is a $y'\in x$ with $y\lesssim^*_Q y'$ and such that $y'$ is multiplicatively closed, or
        \item there is $y\in \hat{I}^*_{\alpha+1}(\prim(Q))$ such that $x\sim^*_Q y$.
    \end{itemize}
\end{Theorem}
\begin{proof}
    Notice that we are within the assumption of \Cref{lem:wqo-set} and \Cref{lem:next-lev-mult-qo}.

    It is practical to first prove an intermediate claim, namely that, under the above assumptions, one of the following \emph{three} options holds for $x$:
    \begin{enumerate}
        \item $x\sim^*_Q({\prim}(\hat{I}^*_\delta(Q))\cap x)^*$ for some $\delta\leq \alpha$, or
        \item\label{item:cofinal-mc} for every $y\in x$ there is a $y'\in x$ with $y\lesssim^*_Q y'$ and such that $y'$ is multiplicatively closed, or
        \item there is $y\in \hat{I}^*_{\alpha+1}(\prim(Q))$ such that $x\sim^*_Q y$.
    \end{enumerate}
    We prove the claim by induction on $\beta\leq\alpha$. If $\beta=0$, the result follows directly from the classification of prime ideals of $Q$.

    Suppose the statement holds for all $\gamma<\beta$, and let $x\in {\prim}(\hat{I}^*_\beta(Q))$. We can suppose that $\beta$ is not a limit ordinal, since no new elements are added at limit stages, say that $\beta=\delta+1$. We know that $x\sim^*_Q ({\prim}(\hat{I}^*_\delta(Q))\cap x)^*$ or $x\sim^*_Q ({\prim}(\hat{I}^*_\delta(Q))\cap x)\downarrow$, by \Cref{theo:id-two-forms}: we can then assume that $x\sim^*_Q ({\prim}(\hat{I}^*_\delta(Q))\cap x)\downarrow$ and that there exists $y\in x$ such that for every $y'\in x$ with $y\lesssim^*_Q y'$, $y'$ is not multiplicatively closed. There are two cases: 
    \begin{itemize}
        \item There exists $y''\in x$ with $y'\lesssim^*_Q y''$ and for all $z\in ({\prim}(\hat{I}^*_\delta(Q))\cap x)$ with $y''\lesssim^*_Q z$, it holds that $z \in Q$. %Notice that it cannot be the case that all such $z$'s are in $Q\setminus \prim(Q)$, since we know that $x\sim^*_Q (\prim(x))\downarrow$. 
        We can then find a cofinal set of elements of $\prim(Q)$ in $x$. But since we are assuming that $\prim(Q)\cup \{e\}$ is downward-closed in $Q$, $x$ cannot contain any element of $Q\setminus \prim(Q)$, and thus $x\in \dot I^*_{\alpha+1}(\prim(Q))$.
        \item The case above fails, and so there is a cofinal set $C$ of $({\prim}(\hat{I}^*_\delta(Q))\cap x)$ constituted by non-urelements. By our assumptions, we can restrict to consider the subset $N\subseteq C$ constituted by non-multiplicatively closed elements. For every element $w\in N$, we can apply the inductive assumption: notice that if $w\sim^*_Q ({\prim}(\hat{I}^*_\varepsilon(Q))\cap w)^*$ for some $\varepsilon$, then $w$ would be multiplicatively closed, so this case cannot occur. We next claim that if $w$ contained a cofinal set of multiplicatively closed elements, then it would itself be multiplicatively closed, and so this case cannot occur either. Indeed, suppose that $a,b\in w$, by directedness we can then find a multiplicatively closed $c$ above both $a$ and $b$. But then, $ab\lesssim^*_Q cc \sim^*_Q c$, so $ab\in w$ since it is downward-closed. 

        It follows that every element of $N$ is in $\hat{I}^*_{\alpha+1}(\prim(Q))$, and so, as in the previous case, we can conclude that $x\in \hat{I}^*_{\alpha+1}(Q)$.
    \end{itemize}

    This concludes the proof of the intermediate claim. In order to get the Theorem, it is sufficient to notice that if $x\in {\prim} (\hat I^*_{\alpha+1}(Q))$ satisfies \Cref{item:cofinal-mc}, then $x\sim^*_Q({\prim}(\hat{I}^*_\delta(Q))\cap x)^*$ for some $\delta\leq \alpha$. But we have already shown that if $x$ satisfies \Cref{item:cofinal-mc}, then it is multiplicatively closed, and so $x\sim^*_Q x^*$. If $\delta$ is the least ordinal such that $x\in \hat{I}^*_{\delta+1}(Q)$, given that every element of $x$ is the finite product of primes, it follows that $x\sim^*_Q({\prim}(\hat{I}^*_\delta(Q))\cap x)^*$.
\end{proof}

We remark that, in general, an element of $\hat{I}^*_{\alpha+1}(\prim(Q))$ can be multiplicatively closed, so both cases of the above Theorem could hold simultaneously.

\section{$P^{<\omega}$ as multiplicative qo}\label{sec:main-theo}

In this section, we apply what we have seen so far to the case where the multiplicative qo of urelements $(Q,\leq_Q,\cdot)$ is $(P^{<\omega},\preceq,\conc)$. As mentioned before, we assume that $P$, and so $P^{<\omega}$, is a set.

The main result is that, if $(P,\leq_P)$ is a bqo, then for every ordinal $\alpha$ there is an order-reflecting map $f:\prim (\dot I^*_\alpha(P^{<\omega})) \to \dot V^*(P\sqcup \{\star\})$, which allows us to conclude that $(P^{<\omega},\preceq)$ is a bqo. The intuitive idea of how to get this result is to apply \Cref{cor:id-mult-or-itprim} at every level, and to map ``large'' ideals of the form $(\prim (I))^*$ into sets obtained using the large element $\star$, and ``small'' ideals arising from iterated ideals of primes to ``small'' elements obtained without using $\star$. 
%\todo{change this}

There are, of course, several difficulties in implementing the idea presented above. One of the most apparent ones is that, in order to apply \Cref{cor:id-mult-or-itprim}, we need to consider ideals built over a set: in general, we do not know that every $\dot I^*_\alpha(P^{<\omega})$ is (isomorphic to) a set if we do not also know that $P^{<\omega}$ is a bqo, which is the very thing we want to prove. Here is where the sets $\hat{I}^*_\alpha(Q)$ defined in the previous section come into play.

Before we proceed with the proof, we enunciate the specific property of $P^{<\omega}$ as a multiplicative qo that we will need in the proof.

\begin{Lemma}\label{lem:pomeg-either-or}
    ($\prsou(\mathsf{ch})$)
    Suppose that $P$ is a set and $\hat{I}^*_\alpha(P^{<\omega})$ is a monoidal wqo$^+$. Consider an element $x\in {\prim}(\hat{I}^*_\alpha(P^{<\omega}))$. Then either $x$ is multiplicatively closed, or there is $y\in \hat{I}^*_\alpha(P)$ such that $x\sim^*_{P^{<\omega}} y$.
\end{Lemma}
\begin{proof}
    By \Cref{cor:id-mult-or-itprim}, we only need to show that if $x$ is in $\hat{I}^*_\alpha(P)$, then $x$ is not multiplicatively closed. Suppose for a contradiction that this is not the case.
    If, for all $p\in P$, $p\not\lesssim^*_{P^{<\omega}} x$ holds, then $x$ is equivalent to the empty string. So this is not the case. By the fact that $x$ is downward-closed and multiplicatively closed, it follows that $pp\in x$, and $pp\not \in P$, contradiction (in essence, we are only using the fact that no element of $P$ is idempotent).
\end{proof}

\begin{Definition}
    ($\prsou$) Let $(P,\leq_P)$ be a qo and let $P^{<\omega}$ be a set. %We denote by $P'$ the downward-closure of $P$ in $P^{<\omega}$, namely $P$ with the addition of the empty string $\epsilon$ as a minimal element.
    We define the primitive set-recursive maps $f_\alpha: \bigcup_{\beta\leq\alpha}{\prim}(\hat{I}^*_\beta(P^{<\omega})) \to \dot V^*_\alpha(P\sqcup \{\star\})\sqcup \{\bot\}$ by recursion as follows (here we read $\bot$ as ``undefined''):
    \begin{itemize}
        \item $f_0(p)=p$ for every $p\in P$.
        \item 
        $
            f_{\alpha+1}(x)= 
            \begin{cases}
                f_\alpha(x) & \text{ if } x\in \dom f_\alpha \\
                \{ f_\alpha(x') : x'\in x\cap \dom f_\alpha\} & \text{ if } \exists y \in \hat{I}^*_{\alpha+1}(P)
                 ( y \sim^*_{P^{<\omega}} x), \\
                & \text{ }x\not\in \dom f_\alpha,\\
                & \text{ and $f_\alpha(x')\ne \bot$ for all }x'\in x\cap \dom f_\alpha\\
                \{ f_\alpha(x'):x'\in x\cap\dom f_\alpha\} \cup \{\star\} 
                & \text{ if the previous cases do not hold,}\\
                & \text{ $f_\alpha(x')\ne \bot$ for all }x'\in x\cap \dom f_\alpha\text{, and}\\
                & \text{ there is multiplicatively closed } y\sim^*_{P^{<\omega}} x\\
                &\text{ such that $y\in \dot{\mathcal{I}}^f({\prim}(\hat{I}^*_\alpha(P^{<\omega})))$,} \\
                \bot & \text{ otherwise.}
            \end{cases}
        $
        \item $f_\lambda = \bigcup_{\alpha<\lambda} f_\alpha 
        % \rst_{\widehat{\prim}(\hat{I}_\alpha(P^{<\omega}))}
        $
    \end{itemize}
\end{Definition}

The map defined above is clearly primitive set-recursive. To see that it is well-defined, we need to notice two things:
\begin{itemize}
    \item The primes of $(P^{<\omega},\leq_P,\conc)$ are indeed just the elements of $P$, as an easy verification shows.
    \item ${\prim}(\hat{I}^*_\lambda(P^{<\omega}))\subseteq \bigcup_{\alpha<\lambda} {\prim}(\hat{I}^*_\alpha(P^{<\omega}))$: indeed, if $x\in \hat{I}^*_\lambda(P^{<\omega})$ is not decomposable into factors from $\hat{I}^*_\lambda(P^{<\omega})$, then it is also not decomposable into factors from $\hat{I}^*_\alpha(P^{<\omega})$, with $\alpha<\lambda$, since $\hat{I}^*_\lambda(P^{<\omega})=\bigcup_{\alpha<\lambda} \hat{I}^*_\alpha(P^{<\omega})$.
    %\todo{need to change def of primes?}
\end{itemize}

\begin{Theorem}\label{theo:ord-refl-id-v}
    ($\prsou(\mathsf{ch})$) Let $\zeta$ be an ordinal, let $P^{<\omega}$ be a set and let $(P,\leq_P)$ be a $\zeta$-wqo. For every ordinal $\alpha<\zeta$, the following hold:
    \begin{enumerate}
        \item\label{item:theo-ran} %For every $x\in \dom f_\alpha$, the support of $f_\alpha(x)$ is contained in $P\sqcup\{\star\}$ and has non-empty intersection with $P$.
        For every $x\in \dom f_\alpha$, $f_\alpha(x)\neq \bot$.
        %\todo{change this according to the new definition}
        \item\label{item:theo-ord-refl} The map $f_\alpha : (\bigcup_{\beta\leq\alpha}{\prim}(\hat{I}^*_\beta(P^{<\omega})),\lesssim^*_{P^{<\omega}}) \to (\dot V^*_\alpha(P\sqcup \{\star\}), \lesssim^*_{P\sqcup\{\star\}})$ is order-reflecting.
        \item\label{item:theo-mult-qo} for all $x,y\in \hat{I}^*_\alpha(P^{<\omega})$, $x \, \hat{\cdot} \, y \sim^*_{P^{<\omega}} xy$, and so $(\hat{I}^*_\alpha(P^{<\omega}), \lesssim^*_{P^{<\omega}}, \hat{\cdot})$ is a multiplicative qo.
        \item\label{item:theo-plus} $(\hat{I}^*_\alpha(P^{<\omega}), \lesssim^*_{P^{<\omega}}, \hat{\cdot})$ has the $+$-property.
        \item\label{item:theo-prime-fac} For every $x\in \hat{I}^*_\alpha(P^{<\omega})$, there is a finite sequence $x_0,\dots,x_n$ of elements of ${\prim}(\hat{I}^*_\alpha(P^{<\omega}))$ such that $x\sim^*_{P^{<\omega}} x_0 \, \hat{\cdot} \, \dots \, \hat{\cdot} \, x_n$.
    \end{enumerate}
\end{Theorem}
\begin{proof}
    Notice that the claim for a fixed ordinal $\alpha$ is $\Delta_0$. We can then prove the theorem by induction on $\alpha$. 

    If $\alpha=0$, all the properties are easily verified, again by the fact that ${\prim}(\hat{I}^*_0(P^{<\omega}))$ is just $P$.

    Suppose that the claim holds for all $\beta<\alpha$. We distinguish between the case that $\alpha=\gamma+1$ for some $\gamma$ and the case that $\alpha$ is a limit. 

    Suppose first that $\alpha=\gamma+1$. We start by showing that $\hat{I}^*_\gamma(P^{<\omega})$ is a wqo, by contradiction. 
    Consider two finite sequences $u$ and $v$ of elements of ${\prim}(\hat{I}^*_\gamma(P^{<\omega}))$, respectively of length $n$ and $m$: by inductive assumption \ref{item:theo-mult-qo}, it is easy to check that if $u\preceq v$, then $u(0) \, \hat{\cdot} \, \dots \, \hat{\cdot} \, u(n) \lesssim^*_{P^{<\omega}} v(0) \, \hat{\cdot} \, \dots \, \hat{\cdot} \, v(m)$. Moreover, by inductive assumption \ref{item:theo-prime-fac}, every element of $\hat{I}^*_\gamma(P^{<\omega})$ is equivalent to the finite product of elements from ${\prim}(\hat{I}^*_\gamma(P^{<\omega}))$. We can then apply Higman's Theorem (available in $\prsou$ by \cite[Theorem 3.10]{nwfr-pakhomov-solda}) to conclude that if $\hat{I}^*_\gamma(P^{<\omega})$ is not a wqo, then neither is ${\prim}(\hat{I}^*_\gamma(P^{<\omega}))$, and by inductive assumption \ref{item:theo-ord-refl}, then neither is $\dot V^*_\gamma(P\sqcup \{\star\})\sqcup \{\bot\}$. But we are assuming that $P$ is a bqo, which implies that $P\sqcup \{\star\}$ is a bqo (by a standard application of the Clopen Ramsey Theorem, available in $\prsou$ by \cite[Theorem 3.11]{nwfr-pakhomov-solda}): thus, by \cite[Theorem 5.12]{nwfr-pakhomov-solda}, $\dot V^*_\gamma(P\sqcup \{\star\})$ is a wqo, which means that so is $\dot V^*_\gamma(P\sqcup \{\star\})\sqcup \{\bot\}$. This contradiction finally proves that $\hat{I}^*_\gamma(P^{<\omega})$ is indeed a wqo. 

    Since $\hat{I}^*_\gamma(P^{<\omega})$ being a wqo implies that $\hat{I}^*_\delta(P^{<\omega})$ is a wqo for all $\delta\leq \gamma$, we can apply \Cref{lem:next-lev-mult-qo} and \Cref{lem:id-mult-qo} to conclude that \ref{item:theo-mult-qo} holds for $\alpha=\gamma+1$, and by the same two lemmas, inductive assumption \ref{item:theo-plus} and \Cref{lem:wqo-plus-to-id-plus} imply that property \ref{item:theo-plus} holds for $\alpha=\gamma+1$. 
    Moreover, \Cref{lem:next-lev-mult-qo} and \Cref{lem:id-mult-qo} also imply that $\hat{I}^*_{\gamma+1}(P^{<\omega})$ is well-founded. Then, \Cref{lem:fin-fact} allows us to conclude that property \ref{item:theo-prime-fac} holds with $\alpha=\gamma+1$.
    %Moreover, \Cref{lem:next-lev-mult-qo} and \Cref{lem:id-mult-qo} have two other consequences: on the one hand, the functions $\widehat{\prim}$ and $\prim$ coincide on $\hat{I}^*_{\gamma+1}(P^{<\omega})$,  and on the other $\hat{I}^*_{\gamma+1}(P^{<\omega})$ is well-founded. Then, \Cref{lem:fin-fact} allows us to conclude that properties \ref{item:theo-prime-fac} holds with $\alpha=\gamma+1$. 

    We can also apply \Cref{cor:id-mult-or-itprim} to $\hat{I}^*_\gamma(P^{<\omega})$ (noticing that the set $P$ is downward-closed in $P^{<\omega}$), thus obtaining that every prime in $\hat{I}^*_{\gamma+1}(P^{<\omega})$ is also an element of $\hat{I}^*_{\alpha+1}(P)$ or it is multiplicatively closed.
    This means that the fourth case in the definition of $f_{\gamma+1}$ is never triggered, and together with inductive assumption \ref{item:theo-ran}, this proves that property \ref{item:theo-ran} holds for $\alpha=\gamma+1$. 
    
    So we are only left to prove that property \ref{item:theo-ord-refl} holds for $\alpha=\gamma+1$. Let $x,y\in {\prim}(\hat{I}^*_{\gamma+1}(P^{<\omega}))$ be such that $f_{\gamma+1}(x)\lesssim^*_{P\sqcup\{\star\}} f_{\gamma+1}(y)$, we want to show that $x\lesssim^*_{P^{<\omega}} y$. We consider several cases:
    \begin{itemize}
        \item If $x,y\in \dom f_\gamma$, we conclude by inductive assumption.
        \item If $x\in \dom f_\gamma$ and $y\not\in \dom f_\gamma$, there are two cases. First, suppose that $\star$ is in the support of $f_{\gamma+1}(x)$: then, $\star$ is in the support of $f_{\gamma+1}(y)$ as well, and so both $x$ and $y$ are multiplicatively closed. Since by inductive assumption every prime of $x$ is equivalent to an element of $y$, we conclude that indeed $x\lesssim^*_{P^{<\omega}}y$. 

        If instead $\star$ is not in the support of $f_{\gamma+1}(x)$, then by \Cref{lem:pomeg-either-or}, %either $x\sim^*_{P^{<\omega}} \epsilon$, in which case it is trivial that $x\lesssim^*_{P^{<\omega}} y$, or 
        $x$ is not multiplicatively closed. But if $x$ is not multiplicatively closed, then it is the downward-closure of its own primes, which are in the domain of $f_\gamma$. We can then conclude that $x\lesssim^*_{P^{<\omega}} y$ by inductive assumption.
        \item If $x\not\in \dom f_\gamma$ and $y\in \dom f_\gamma$, %we immediately conclude by inductive assumption. 
        the claim follows from an argument similar to the previous one.
        \item Finally, if $x,y\not\in \dom f_\gamma$, we have again to distinguish between the two cases that $x\in \hat{I}^*_{\gamma+1}(P)$ and that $x$ is multiplicatively closed. The former case is easily dealt with by inductive assumption, whereas for the second we just have to notice that if $x$ is multiplicatively closed, then $\star$ is in the support of $f_{\gamma+1}(x)$, which means that $\star$ is in the support of $f_{\gamma+1}(y)$, and hence $y$ is multiplicatively closed as well. Hence both $x$ and $y$ are determined by their primes, and the claim follows by inductive assumption.
        %All of these cases are dealt with easily by inductive assumption. \todo{expand?}
    \end{itemize}

    So we are left with the case of $\alpha=\lambda$ a limit ordinal. In this case, properties \ref{item:theo-ran} and \ref{item:theo-ord-refl} hold thanks to the corresponding inductive assumptions. Since we can conclude that $\hat{I}^*_\gamma(P^{<\omega})$ is a wqo for every $\gamma<\lambda$, we can apply \Cref{lem:wqo-set} to conclude that property \ref{item:theo-mult-qo} holds. Property \ref{item:theo-plus} then also holds by inductive assumption. Finally, we know that every element of $\hat{I}^*_\lambda(P^{<\omega})$ is the product of finitely many elements of $\bigcup_{\gamma<\lambda}{\prim}(\hat{I}^*_\gamma(P^{<\omega}))$: since $\lambda<\zeta$ we can conclude that this union is a wqo, since it order-reflectingly embeds into $\dot V^*_\lambda(P\sqcup \{\star\})$, which we know to be a wqo. So, again by Higman's Theorem, we conclude that $\hat{I}^*_\lambda(P^{<\omega})$ is a wqo. We can then apply \Cref{lem:fin-fact} to conclude that every element of $\hat{I}^*_\gamma(P^{<\omega})$ is the finite product of primes of $\hat{I}^*_\gamma(P^{<\omega})$, thus establishing property \ref{item:theo-prime-fac}.
\end{proof}

\begin{Remark} 
Arguing without a specific base theory, we observe that for every $\zeta$, if $P$ is a $\zeta$-wqo, then, for every $\alpha<\zeta$, $f_\alpha\colon \bigcup_{\beta\leq\alpha}\prim(\dot I^*_\beta(P^{<\omega}))\to \dot V_\alpha^*(P\sqcup \{\star\})$ is not only order-reflecting, but order-preserving as well, as one easily sees by induction. We did not include this fact in the proof above as it is not needed to reach the desired conclusion.
\end{Remark}

\begin{Corollary}\label{cor:id-hig-wqo}
    ($\prsou(\mathsf{ch})$) Let $\zeta$ be an ordinal, let $P^{<\omega}$ be a set and let $(P,\leq_P)$ be a $\zeta$-wqo. Then, for every ordinal $\alpha<\zeta$, $\hat I^*_\alpha(P^{<\omega})$ is a wqo.
\end{Corollary}
\begin{proof}
    By property \ref{item:theo-ord-refl}, the primes of $\hat I^*_\alpha(P^{<\omega})$ form a wqo. Since every element of $\hat I^*_\alpha(P^{<\omega})$ is a finite product of primes by property \ref{item:theo-prime-fac}, it follows that $\hat I^*_\alpha(P^{<\omega})$ is a wqo by Higman's Theorem.
\end{proof}

Using the observations made in \cite[Section 3.3]{nwfr-pakhomov-solda}, we can derive the following result.

\begin{Corollary}\label{cor:ght-atr}
    The Generalized Higman's Theorem is provable in $\atr$.
\end{Corollary}
\begin{proof}
    We first argue in $\prsou(\mathsf{enu})$, which is stronger than $\prsou(\mathsf{ch})$. Suppose that $P$ is a bqo, then from \Cref{cor:id-hig-wqo} it follows that $\hat I^*_\alpha(P^{<\omega})$ is a wqo for every $\alpha$, and then so is $\dot I^*_\alpha(P^{<\omega})$ by \Cref{lem:wqo-set}. We conclude that $P^{<\omega}$ is a bqo by \Cref{theo:ideal-wqo-qo-bqo}.

    By the conservativity results given in \cite[Section 3.3]{nwfr-pakhomov-solda}, this proof in $\prsou(\mathsf{enu})$ yields a proof of the same fact in $\atr$.
\end{proof}

\section{On comparing elements of $\dot I^*_\alpha(P^{<\omega})$}\label{sec:comparing-el-gen-hig}

%\todo[inline]{Mention the French}

In this section, we provide a more precise characterization of the orders $\dot I^*_\alpha(P^{<\omega})$. Since we do not apply these constructions to questions of reverse mathematics, for the convenience of the readers we will no longer be verifying the formalizability of our constructions in $\prsou(\mathsf{ch})$. Thus we work in an unspecified set-theoretic foundation. We remark that we nonetheless do not expect any fundamental difficulties to arise with such a formalization.

Let us be more specific about our goals. We show that the order $\lesssim^*_{P^{<\omega}}$ on $\dot I^*_\alpha(P^{<\omega})$ behaves quite similarly to the Higman ordering on $P^{<\omega}$: namely, we show that, if $x_0,\dots,x_n,y_0,\dots,y_m$ are primes of $\dot I^*_\alpha(P^{<\omega})$, then $x_0\cdot\ldots\cdot x_n \lesssim^*_{P^{<\omega}} y_0\cdot\ldots \cdot y_m$ happens exactly when there is a weakly increasing function $h:\{0,\dots,n\} \to \{0,\dots,m\}$ such that $x_i\leq_Q y_{h(i)}$, and two or more $i$'s are mapped to the same $j$ only if $y_j$ is multiplicatively closed. Moreover, an analysis of how primes are produced allows us to be quite precise about the generators of this generalized Higman ordering. 

We remark that the content of this section is once again inspired by \cite{ideal-decomp-goubault-larrecq-halfon-karakndikar-narayan-schnoebelen}: in a sense, what we do here is a generalization of Lemma 4.13 in that paper.

%We add this result in order to further highlight the nice behavior of the classes $\dot I^*_\alpha(P^{<\omega})$ when seen as monoidal qos and the structural properties they inherit from $P^{<\omega}$, but it is not needed for reverse mathematical considerations. We are therefore much more lapse with the theory over which we give the proof.

We start by giving a characterization of the primes of $\dot I^*_\alpha(P^{<\omega})$.

\begin{Lemma}\label{lem:xy-wz}
    Let $(Q,\leq_Q,\cdot)$ be a monoidal qo with the property that, for every $x,y,w,z\in Q$, if $xy\leq_Q wz$ then $x\leq_Q w$ or $y\leq_Q z$.
    Then, for every $X,Y,W,Z\in \dot{\mathcal{I}}(Q)$, $XY\subseteq WZ$ implies that $X\subseteq W$ or $Y\subseteq Z$.
\end{Lemma}
\begin{proof}
    %Let $\{x_i:i\in\omega\}$ be a cofinal sequence in $X$, and $\{y_i:i\in\omega\}$ a cofinal sequence in $Y$. It is easily seen that $x_iy_i$ is then a cofinal sequence in $XY$. Since every element of $XY$ is below an element of $WZ$, we have that for every $i\in\omega$ $x_iy_i\leq_Q w_iz_i$, for some $w_i\in W$ and $z_i\in Z$, and so $x_i\leq_Q w_i$ or $y_i\leq_Q z_i$. One of the two cases happens for infinitely many $i$'s: if it is the former, we can conclude that $X\subseteq W$, otherwise that $Y\subseteq Z$.
    We prove the claim by contradiction: suppose there are ideals $X,Y,W,Z$ such that $XY\subseteq WZ$ but $X\not\subseteq W$ and $Y\not\subseteq Z$. Then there are $x_0\in X$ and $y_0\in Y$ such that $x_0\not\in W$ and $y_0\not\in Z$. But $x_0y_0\in WZ$, and so there are $w_0\in W$ and $z_0\in Z$ such that $x_0y_0\leq_Q w_0z_0$. From the assumed property of $Q$ we immediately derive the desired contradiction, since $W$ and $Z$ are downward-closed.
\end{proof}

\begin{Corollary}\label{cor:xy-wz-itid}
    For every $\alpha$ and every $x,y,w,z\in \dot I^*_\alpha(P^{<\omega})$, if $xy\lesssim^*_{P^{<\omega}} wz$, then $x\lesssim^*_{P^{<\omega}} w$ or $y\lesssim^*_{P^{<\omega}} z$.
\end{Corollary}
\begin{proof}
    An easy induction on $\alpha$, arguing as in \Cref{lem:xy-wz} and noting that products at limit levels can be seen as products at lower levels.
\end{proof}

\begin{Corollary}\label{cor:stars-are-primes}
For every $X\subseteq \dot I_\alpha^*(P^{<\omega})$, the set $X^*$ is a prime in $\dot I_{\beta}^*(P^{<\omega})$, for any $\beta>\alpha$. 
\end{Corollary}
\begin{proof}
    Note that $X^*$ is idempotent, i.e. $X^* X^* \sim^*_{P^{<\omega}} X^*$. Assume for a contradiction that $X^*\sim_{P^{<\omega}}^* YZ$ for some $Y,Z \in  \dot I_{\beta}^*(P^<\omega)$ such that $Y \lnsim X^*$ and $Z \lnsim X^*$. Since $X^*X^*\sim_{P^{<\omega}}^* YZ$ by \Cref{cor:xy-wz-itid} we should have either $X^* \lesssim Y$ or $X^* \lesssim Z$, contradicting the assumption about the choice of $Y$ and $Z$.
\end{proof}

\begin{Lemma}\label{lem:non-idmp-primes-description} All elements of $\dot I_\alpha^*(P)$ are non-idempotent primes in $\dot I_\alpha^*(P^{<\omega})$.\end{Lemma}
\begin{proof}
    Notice that the lemma is immediately implied by the fact that, for any $X,Y\in \dot I_\alpha^*(P^{<\omega})$ strictly above $\epsilon$, the product $XY$ is not equivalent to any element in $\dot I_\alpha^*(P)$. To see that this fact holds we only observe that the supports of both $X$ and $Y$ will respectively contain some non-$\epsilon$ strings $x\in P^{<\omega}$ and $y\in P^{<\omega}$. Therefore the support of the product $XY$ will contain a string $xy$ of length at least $2$, which is not $\le_{P^{\omega}}$-below any element of $P$. Hence the product $XY$ is not equivalent to any element of $\dot I_\alpha^*(P)$.
\end{proof}

\begin{Definition}
    Suppose $P$ is a qo. Let the hierarchy of sets $S_\alpha(P)\subseteq \dot I_\alpha^*(P^{<\omega})$ be defined as follows:
    \begin{enumerate}
        \item $S_0(P)=\emptyset$;
        \item $S_{\alpha+1}(P)=S_{\alpha}(P)\cup \{X^*\mid  X\in \dot {\mathcal{D}} (S_\alpha(P)\cup \dot I _\alpha^*(P))\}$;
        \item $S_\lambda(P)=\bigcup \limits_{\alpha<\lambda} S_\alpha(P)$.
    \end{enumerate}
\end{Definition}

\begin{Theorem}\label{theo:primes_description} Suppose $P$ is an $(\alpha+1)$-wqo. The sets $\dot I_\alpha^*(P)$ and $S_\alpha(P)$ are up to equivalence the sets of non-idempotent and idempotent primes of $\dot I_\alpha^*(P^{<\omega})$, respectively. 

More precisely:
\begin{enumerate}
    \item \label{item:pd_1} $\dot I_\alpha^*(P)$ consists of non-idempotent primes of $\dot I_\alpha^*(P^{<\omega})$,
    \item \label{item:pd_2} $S_\alpha(P)$ consists of idempotent primes of $\dot I_\alpha^*(P^{<\omega})$,
    \item \label{item:pd_3} every non-idempotent prime of $\dot I_\alpha^*(P^{<\omega})$ is equivalent to an element of $\dot I_\alpha^*(P)$,
    \item \label{item:pd_4} every idempotent prime of $\dot I_\alpha^*(P^{<\omega})$ is equivalent to an element of $S_\alpha(P)$.
\end{enumerate}
\end{Theorem}
\begin{proof}
    Clearly, \Cref{item:pd_1} is \Cref{lem:non-idmp-primes-description} and \Cref{item:pd_2} is implied by \Cref{cor:stars-are-primes}. 

    \Cref{item:pd_3} is immediately implied by \Cref{lem:pomeg-either-or}. To get \Cref{item:pd_4} from \Cref{lem:pomeg-either-or} it is enough to show that every multiplicatively closed prime $X$ is equivalent to an element of $S_{\alpha}(P)$.
    
    We prove this by showing by induction on $\beta< \alpha$ that every multiplicatively closed prime $X$ of $\dot I_{\beta+1}^*(P^{<\omega})$ is equivalent to an element of $S_{\beta+1}(P)$. Namely we consider the set $R$ of all $r\in S_\beta(P)\cup \dot I_\beta^*(P)$ such that $r\lesssim^*_{P^{<\omega}}x$ for some $x\in X$. From the inductive assumption it already follows that each prime of $\dot I_\beta^*(P^{<\omega})$ is either equivalent to an element of $\dot I_\beta^*(P)$ or to an element of $S_\beta(P)$. Therefore, for any $x\in X$ given its representation as a product of primes the factors can be equivalently transformed to elements of $R$. Thus we see that $R^*\supset X$. At the same time the multiplicative closedness of $X$ implies that $X\supset R^*$ and hence $X=R^*$, which finishes the proof of \Cref{item:pd_4}.
\end{proof}

Next, we introduce the notion of an abstractly Higman ordering, and show that this property is passed from a wqo$^+$ $Q$ to its space of ideals.

\begin{Definition}
We call a monoidal wfqo \emph{abstractly Higman} if it satisfies the following property:
    \begin{enumerate}
    %\item  $\prim(Q)\cup \{e\}$ (where $e$ is the identity) is a downward closed set.
    %\item\label{item:gen-hig-1} If $p,q$ are primes with $p\leq_Q q$ and $p$ is idempotent, then so is $q$ (i.e., if $pp \equiv p$ and $p\leq_Q q$, then $qq\equiv q$).
    \item\label{item:gen-hig} for every sequence of primes $p_0,\dots,p_n,q_0,\dots,q_m$, $p_0\cdot\ldots\cdot p_n \leq_Q q_0\cdot\ldots\cdot q_m$ holds if and only if there is a weakly increasing function $h:\{0,\dots,n\}\to \{0,\dots,m\}$ such that  $p_i\leq_Q q_{h(i)}$ for all $i\le m$, and moreover, if $|h^{-1}(j)|>1$, then $q_{j}$ is idempotent.
    \end{enumerate}
\end{Definition}

\begin{Theorem}\label{theo:gen-hig-next-lev}
    Let $(Q,\leq_Q,\cdot)$ be an abstractly Higman monoidal wqo$^+$. Then the qo $(\dot{\mathcal{I}}(Q),\subseteq, \bullet)$ is an abstractly Higman monoidal wfqo.
\end{Theorem}
\begin{proof}
    First, notice that the right-to-left implication in property \ref{item:gen-hig} is trivial. %\todo{more detail?}
    So we are only left to prove the left-to-right implication.

    We start by noticing that $Q$ being abstractly Higman implies that if $xy\leq_Q wz$, then $x\leq_Q w$ or $y\leq_Q z$: indeed, if we express $x,y,w,z$ as products of primes, then the function $h$ can be used to determine which of the two cases holds. 

    Next, we notice that property \ref{item:gen-hig} implies that if $p,q$ are primes with $p\leq_Q q$ and $p$ is idempotent, then so is $q$ (i.e., if $pp \equiv p$ and $p\leq_Q q$, then $qq\equiv q$). Indeed, suppose this was not the case: then $pp\equiv p \leq_Q q$, but then the map given by property \ref{item:gen-hig} witnessing $pp\leq_Q q$ should give that $q$ is idempotent, contradiction.

    We now prove the claim of the theorem, by induction on $n$. If $n=0$, then the claim is that for every $P_0,R_0,\dots,R_m$ primes of $\dot{\mathcal{I}}(Q)$, if $P_0\subseteq R_0\bullet\dots\bullet R_m$, then there is $j\leq m$ with $P_0\subseteq R_j$: this follows from \Cref{lem:qo+-prod}. 

    Suppose the claim holds for $n$, we prove it for $n+1$. Suppose for a contradiction that the claim is false: there are primes $P_0,\dots,P_{n+1},Q_0,\dots,Q_m$ such that $P_0\bullet \dots \bullet P_{n+1} \subseteq R_0 \bullet \dots \bullet R_m$ but there is no weakly increasing function as in property \ref{item:gen-hig}. Again by \Cref{lem:qo+-prod}, there is $j\leq m$ with $P_{n+1}\subseteq R_j$: let $k$ be the largest of such $j$'s. We can apply \Cref{lem:xy-wz} to $P_0\bullet\dots\bullet P_n, P_{n+1}, R_0\bullet\dots \bullet R_k, R_{k+1}\bullet\dots\bullet R_m$ to obtain that $P_0\bullet\dots\bullet P_n\subseteq R_0\bullet\dots \bullet R_k$. By inductive assumption, there is a weakly increasing map $f$ from $\{0,\dots,n\}$ to $\{0,\dots,m\}$ as given by property \ref{item:gen-hig}. If $f(n)<k$, we are done, since then the map obtained by adding the pair $(n+1,k)$ to $f$ implies the conclusion, so we can suppose that $f(n)=k$. In particular, we can also assume that $P_0\bullet\dots\bullet P_n \not\subseteq R_0\bullet\dots\bullet R_{k-1}$, in the case that $k\neq 0$.
    
    Moreover, if $R_k$ is multiplicatively closed, we are again done by adding $(n+1,k)$ to $f$. So we are only left with the case that $R_k$ is not multiplicatively closed. By \Cref{theo:id-two-forms}, it follows that $R_k=(\prim(Q)\cap R_k)\downarrow$. But then, notice that $R_k$ cannot contain any idempotent prime: suppose for a contradiction it did, then by the assumption that any prime above an idempotent is idempotent, and by the fact that ideals are upward-directed, we can conclude that $R_k$ is the downward closure of a set of idempotent primes. Then, let $a,b$ be any two elements of $R_k$, it follows that there is an idempotent prime $p\in R_k$ above both $a$ and $b$: but then $ab\leq_Q pp \leq_Q p$, which would imply that $R_k$ is multiplicatively closed.

    We can thus pick an element $a\in P_0\bullet\dots\bullet P_n$ such that $a\not \in R_0\bullet\dots\bullet R_{k-1}$ (pick any element of $P_0\bullet\dots\bullet P_n$ if $k=0$), and an element $b\in P_{n+1}$ such that $b\not\in R_{k+1}\bullet\dots\bullet R_m$ (pick any element of $P_{n+1}$ if $k=m$). Notice that, thanks to the assumption that $P_{n+1}\subseteq R_k$ and that $R_k$ is not multiplicatively closed, we can take $b$ to be a prime of $R_k$ (this is not really important, but it does simplify the notation in what follows). We can write $a$ as a product of primes $p_0\cdot\ldots\cdot p_s$, so that $ab\equiv p_0\dots p_s b$. We know that $ab\in R_0\bullet\dots\bullet R_m$, so for every $j\leq_Q m$ there is $r_j\in R_j$ such that 
    \[
        p_0\dots p_s b \leq_Q r_0\dots r_m =: r
    \]
    Notice that we can assume that $r_k$ is just a non-idempotent prime: indeed, we have shown above, $R_k$ is the downward closure of its own primes, and none of them is idempotent.

    Every $r_j$ can be written as a product of primes, say that $r_j\equiv q^j_0\dots q^j_{t_j}$, so, considering what we said above, we have that
    %there are primes $q^j_0,\dots,q^j_{t_j}$ such that 
    \[
        p_0\dots p_s b \leq_Q (q^0_0\dots q^0_{t_0}) \dots (q^{k-1}_0\dots q^{k-1}_{t_{k-1}}) \cdot r_k \cdot (q^{k+1}_0\dots q^{k+1}_{t_{k+1}}) \dots (q^m_0\dots q^m_{t_m} )
    \]
    (where we can see all the elements to the left of $r_k$ as the identity if $k=0$, and similarly to the right if $k=m$).
    So by property \ref{item:gen-hig} there is a weakly increasing function $h$ mapping the $p_i$'s and $b$ into the $q^j_{\ell_j}$'s. Then there are two cases:
    \begin{itemize}
        \item if $b$ is mapped to $r_k$ or or any $q^j_l$ with $j<k$, then it follows that $h$ maps $p_0,\dots, p_s$ into $q^j_\ell$'s with $j<k$, which would imply that $a\in R_0\bullet\dots\bullet R_{k-1}$, contradiction, or
        \item $b$ is mapped to a $q^j_\ell$ with $j>k$, which then entails that $b\in R_{k+1}\bullet \dots\bullet R_m$, again a contradiction.
    \end{itemize}
\end{proof}

We can now abstract from the previous result to get closer to the situation where $Q$ is actually the Higman ordering on a qo $P$.

\begin{Definition}
    Let $(P,\leq_P)$ be a qo, $P_n\subseteq P$ be downward-closed in $P$ and $P_i=P\setminus P_n$. We call \emph{generalized Higman ordering on $P,P_n,P_i$} the order $(P^{<\omega},\leq_H)$, where $(p_0,\ldots,p_{n-1})\le_H (q_0,\ldots,q_{m-1})$ holds if and only if there is a weakly increasing $f\colon n\to m$ such that $p_i\le_P q_{f(i)}$ for each $i<n$ and for every $j<m$, if $|f^{-1}(j)|>1$ then $q_j\in P_i$.

    We denote by $H(P,P_n,P_i)$ the monoidal qo obtained endowing the generalized Higman ordering on $P,P_n, P_i$ with the concatenation operation $\conc$ (the neutral element being the empty sequence $\epsilon$).  
\end{Definition}

\begin{Remark}\label{rem:obv-prop-gen-hig}
    We include some easy observations about $H(P,P_n,P_i)$ that will nevertheless come in handy later on.
    \begin{itemize}
        \item It is easily verified that $\leq_H$ is a reflexive and transitive relation, and that $\conc$ is indeed an associative, weakly increasing and monotone operation. Since every element of $H(P,P_n,P_i)$ is a concatenation of elements of $P$, the $+$-property also holds.
        \item We notice that the primes of $H(P,P_n,P_i)$ are the elements of $P$ and those that are $\equiv_H$-equivalent to them. Indeed, suppose that $q_0\conc q_1$ is a prime, for some $q_0,q_1\in P^{<\omega}$: then, without loss of generality, we can assume that $q_0\conc q_1\leq_H q_0$. If $q_0\in P$, we can conclude that $q_0\conc q_1\equiv_H q_0$, otherwise we iterate this procedure with $q_0$ in place of $q_0\conc q_1$: since we are dealing with finite sequences of elements of $P$, this procedure eventually terminates, and gives the desired conclusion. 
        \item Finally, notice that the idempotent primes of $H(P,P_n,P_i)$ are exactly the elements equivalent to an element of $P_i$. Indeed, it follows immediately from the definition that if $q\equiv_H p$ for $p\in P_i$, then $q\conc q \leq_H p$. Conversely, if $q\equiv_H r$ for $r\in P_n$, we cannot have that $q\conc q\leq_H r$, since $r\not\in P_i$.
    \end{itemize}
\end{Remark}

Putting the definition and the previous remarks together, we can conclude that whenever $P$ is a wfqo, $H(P,P_n,P_i)$ is always an abstractly Higman wfqo.

\begin{Lemma}\label{lem:iso-prop-hig}
    Let $(Q,\leq_Q,\cdot)$ be an abstractly Higman wfqo. Let us denote by $P_i$ the set of idempotent primes of $Q$, and by $P_n$ the set of non-idempotent primes. Then the monoidal wfqos $H(P_i\cup P_n,P_n,P_i)/{\equiv_H}$ and $Q/{\equiv_Q}$ are isomorphic.
\end{Lemma}
\begin{proof}
    We have already observed that the primes of $H(P_i\cup P_n,P_n,P_i)$ are the elements equivalent to elements of $P_i\cup P_n$. We then we define an isomorphism $\psi$ from $H(P_i\cup P_n,P_n,P_i)/{\equiv_H}$ to $Q/{\equiv_Q}$ as follows. We put $\psi$ to be the identity on equivalence classes from $P_n\cup P_i$, and for every element $x=p_0\conc \dots \conc p_{n-1}$ of $H(P_i\cup P_n, P_n,P_i)$, we set $\psi([x]_{\equiv_H})=[p_0\cdot\ldots\cdot p_{n-1}]_{\equiv_Q}$. Property \ref{item:gen-hig} guarantees that the map is order-reflecting and order-preserving, and it is easily seen to be surjective.
\end{proof}

We are now ready to give a description of $\dot I^*_\alpha(P^{<\omega})$ as a generalized Higman ordering.

%From \Cref{lem:prim-preserved} and \Cref{theo:primes_description} we immediately derive the following characterization of $\dot I^*_\alpha(P^{<\omega})$:

\begin{Corollary}\label{cor:I_alpha_gen_Higman} 
Suppose $P$ is an $(\alpha+1)$-wqo. Then the monoidal wqo $\dot I^*_\alpha(P^{<\omega})/\sim^*_{P^{<\omega}}$ is isomorphic to $H( \dot I_\alpha^*(P)\cup S_\alpha(P),\dot I_\alpha^*(P), S_\alpha(P))$.
\end{Corollary}
\begin{proof}
    We prove the claim by induction on $\alpha$. If $\alpha=0$, the claim is obvious. If $\alpha=\beta+1$ for some $\beta$, then we know by inductive assumption that $\dot I_\beta(P^{<\omega})/\sim^*_{P^{<\omega}}$ is isomorphic to $H( \dot I_\beta^*(P)\cup S_\beta(P),\dot I_\beta^*(P), S_\beta(P))$, so in particular it is an abstractly Higman ordering. We can thus use \Cref{theo:gen-hig-next-lev} to conclude that $\dot I^*_\alpha(P^{<\omega})$ is abstractly Higman, and then a combination of \Cref{theo:primes_description} and \Cref{lem:iso-prop-hig} gives the desired conclusion. 

    Finally, if $\alpha$ is a limit, we only have to notice that it is abstractly Higman because it inherits the comparisons from the previous levels, and then conclude again by \Cref{theo:primes_description} and \Cref{lem:iso-prop-hig}.
\end{proof}

\section{Conclusions, remarks, and further work}\label{sec:concl}
%----two extra roads to the theorem----
\subsection{Alternative proofs of the main result}
One could imagine several other avenues to prove the main result of this paper. For instance, it should be possible to use the perspective offered in \Cref{sec:comparing-el-gen-hig}, that is, to provide a more direct proof of \Cref{cor:I_alpha_gen_Higman}. For this, the step of induction can either be provided using our methods from \Cref{sec:id-combinatorics} or the previously known \cite{id-higman-kabil-pouzet,ideal-decomp-goubault-larrecq-halfon-karakndikar-narayan-schnoebelen} techniques for the characterization of ideals in Higman orderings. However, we envision the justification of the wqo-ness of $\prim(\dot I^*_\alpha(P^{<\omega}))$ here to be broadly analogous to what we have done in the present paper and to be based upon the construction of order-reflecting maps to $\dot V(P\sqcup \{\star\})$.

Another modification of our proof can be given using the results of Kabil and Pouzet in \cite{id-higman-kabil-pouzet} instead of our developments from \Cref{sec:id-combinatorics}. We note that they used rather different techniques to show that prime ideals of monoidal wqos only have two shapes (that is a more general version of our \Cref{theo:id-two-forms}). Namely, their proof for the general case relies on the characterization of prime ideals for the special case of Higman orderings, while our approach to the characterization of prime ideals essentially goes via the study of properties that a prime ideal should have in order to eventually narrow down the options to only two possible classes. The approaches to the characterization of prime ideals in Higman orderings in \cite{id-higman-kabil-pouzet} and \cite{ideal-decomp-goubault-larrecq-halfon-karakndikar-narayan-schnoebelen} follow quite different lines comparing to \Cref{sec:id-combinatorics} and attain the classification by demonstrating that two possible classes of prime ideals are enough to generate all downward-closed sets generated by finitely many forbidden elements. Although for the cases of Higman orderings and monoidal wqos, these two approaches are quite comparable complexity-wise, we see a possibility that our approach might be better suited to a generalization to the classification of ideals in more complex orderings such as Kruskal orderings (with labels forming a wqo).

\subsection{The orders $\dot I^*_\alpha(Q)$}
In this paper, we introduced the iterated ideal orderings $\dot I^*_\alpha(P)$. But outside of the main result of the paper around Higman orderings, we have not done much in-depth investigation of their structure. The cases of some very simple $P$ were discussed at the end of \Cref{sec:defin}, but it would be interesting to study these orders in more detail, in line with what has already been done for the orders $\dot V^*_\alpha(P)$ \cite{3-bqo-freund,phd-manca,freund-provable}.

\subsection{Bqo-preserving operators and ideals}
In this paper, we observed that, speaking loosely, for a bqo $P$, the structure of $\dot I^*_\alpha(P^{<\omega})$ is roughly as complex as $\dot V^*_\alpha(P)$. A natural question here is whether a similar phenomenon happens for other bqo-preserving operators such as $P\mapsto P^n$ (for a fixed $n$) or $P\mapsto T(P)$, where $T(P)$ is the Kruskal ordering. Our conjecture is that this phenomenon is specific to bqo-preserving operators of roughly the same ``strength'' as the Higman operator.

\bibliographystyle{amsplain}
\bibliography{References}

@article{freund-provable,
  title={Provable Better-Quasi-Orders},
  author={Freund, Anton and Marcone, Alberto and Pakhomov, Fedor and Sold{\`a}, Giovanni},
  journal={Notre Dame Journal of Formal Logic},
  volume={66},
  number={2},
  pages={175--188},
  year={2025},
  publisher={Duke University Press}
}

@phdthesis{phd-manca,
  title={At the limits of predicativity: the reverse mathematics of ordering relations},
  author={Manca, Davide},
  year={2025},
  school={Universit{\"a}t W{\"u}rzburg}
}

@book{book-simpson, place={Cambridge}, edition={2}, series={Perspectives in Logic}, title={Subsystems of Second Order Arithmetic}, DOI={10.1017/CBO9780511581007}, publisher={Cambridge University Press}, author={Simpson, Stephen G.}, year={2009}, collection={Perspectives in Logic}}

@inproceedings{nash-williams-marcone,
author = {Marcone, Alberto},
title = {On the Logical Strength of {N}ash-{W}illiams' Theorem on Transfinite Sequences},
year = {1996},
isbn = {0198538626},
publisher = {Clarendon Press},
address = {USA},
booktitle = {Logic: From Foundations to Applications: European Logic Colloquium},
pages = {327–351},
numpages = {25}
}

@Book{theory-relations-fraisse,
author = { Fraïssé, Roland. },
title = { Theory of relations  (translated by {P}. {C}lote) },
isbn = { 0444878653 },
publisher = { North-Holland ; Sole distributors for the U.S.A. and Canada, Elsevier Science Pub. Co Amsterdam ; New York : New York, N.Y., U.S.A },
pages = { xii, 397 p. ; },
year = { 1986 },
type = { Book },
language = { English },
subjects = { Set theory. },
life-dates = { 1986 -  },
catalogue-url = { https://nla.gov.au/nla.cat-vn380254 },
}

@article {nwt-nash-williams,
    AUTHOR = {Nash-Williams, C. St. J. A.},
     TITLE = {On better-quasi-ordering transfinite sequences},
   JOURNAL = {Proc. Cambridge Philos. Soc.},
  FJOURNAL = {Proceedings of the Cambridge Philosophical Society},
    VOLUME = {64},
      YEAR = {1968},
     PAGES = {273--290},
      ISSN = {0008-1981},
   MRCLASS = {04.60 (05.00)},
  MRNUMBER = {221949},
MRREVIEWER = {W. T. Tutte},
       DOI = {10.1017/s030500410004281x},
       URL = {https://doi.org/10.1017/s030500410004281x},
}

@phdthesis{phd-pequignot,
	Author = {Yann Pequignot},
	School = {Université Paris Diderot - Université de Lausanne},
	Title = {Better-quasi-order: ideals and spaces},
	Type = {PhD thesis},
	Year = {2015}}

@Incollection{fraisse-conj-shore,
author="Shore, Richard A.",
editor="Crossley, John N.
and Remmel, Jeffrey B.
and Shore, Richard A.
and Sweedler, Moss E.",
title="On the strength of Fra{\"i}ss{\'e}'s conjecture",
bookTitle="Logical Methods: In Honor of Anil Nerode's Sixtieth Birthday",
year="1993",
publisher="Birkh{\"a}user Boston",
address="Boston, MA",
pages="782--813",
isbn="978-1-4612-0325-4",
doi="10.1007/978-1-4612-0325-4_26",
url="https://doi.org/10.1007/978-1-4612-0325-4_26"
}

@misc{min-bad-freund-pakhomov-solda,
      title={The logical strength of minimal bad arrays}, 
      author={Anton Freund and Fedor Pakhomov and Giovanni Soldà},
      year={2023},
      eprint={2304.00278},
      archivePrefix={arXiv},
      primaryClass={math.LO}
}

@article{wq-trees-nash-williams,
title={On well-quasi-ordering infinite trees},
volume={61},
DOI={10.1017/S0305004100039062},
number={3},
journal={Mathematical Proceedings of the Cambridge Philosophical Society},
publisher={Cambridge University Press},
author={Nash-Williams, C. St. J. A.},
year={1965},
pages={697–720}}

@article{3-bqo-freund,
  title = {On the logical strength of the better quasi order with three elements},
  ISSN = {1088-6850},
  url = {http://dx.doi.org/10.1090/tran/8966},
  DOI = {10.1090/tran/8966},
  journal = {Transactions of the American Mathematical Society},
  publisher = {American Mathematical Society (AMS)},
  author = {Freund,  Anton},
  volume = {376},
  year = {2023},
  pages = {6709-6727},
  month = jun 
}

@article{fraisee-conj-laver,
 ISSN = {0003486X},
 URL = {http://www.jstor.org/stable/1970754},
 author = {Richard Laver},
 journal = {Annals of Mathematics},
 number = {1},
 pages = {89--111},
 publisher = {Annals of Mathematics},
 title = {On Fraisse's Order Type Conjecture},
 urldate = {2023-07-06},
 volume = {93},
 year = {1971}
}

@phdthesis{phd-freund,
	Author = {Freund, Anton},
	School = {The University of Leeds},
	Title = {Type-Two well-ordering principles, admissible sets, and {$\Pi^1_1$}-comprehension},
	Type = {PhD thesis},
	Year = {2018}}

@incollection {ideal-decomp-goubault-larrecq-halfon-karakndikar-narayan-schnoebelen,
    AUTHOR = {Goubault-Larrecq, Jean and Halfon, Simon and Karandikar,
              Prateek and Narayan Kumar, K. and Schnoebelen, Philippe},
     TITLE = {The ideal approach to computing closed subsets in
              well-quasi-orderings},
 BOOKTITLE = {Well-quasi orders in computation, logic, language and
              reasoning---a unifying concept of proof theory, automata
              theory, formal languages and descriptive set theory},
    SERIES = {Trends Log. Stud. Log. Libr.},
    VOLUME = {53},
     PAGES = {55--105},
 PUBLISHER = {Springer, Cham},
      YEAR = {[2020] \copyright 2020},
      ISBN = {978-3-030-30228-3; 978-3-030-30229-0},
   MRCLASS = {06A07 (03E04)},
  MRNUMBER = {4291878},
MRREVIEWER = {Robert\ Brignall},
       DOI = {10.1007/978-3-030-30229-0\_3},
       URL = {https://doi.org/10.1007/978-3-030-30229-0_3},
}

@article{open-questions-montalban,
 ISSN = {10798986},
 URL = {http://www.jstor.org/stable/41228534},
 author = {Antonio Montalb\'{a}n},
 journal = {The Bulletin of Symbolic Logic},
 number = {3},
 pages = {431--454},
 publisher = {[Association for Symbolic Logic, Cambridge University Press]},
 title = {Open questions in reverse mathematics},
 urldate = {2023-11-23},
 volume = {17},
 year = {2011}
}

@misc{nwfr-pakhomov-solda,
      title={On Nash-Williams' Theorem regarding sequences with finite range}, 
      author={Fedor Pakhomov and Giovanni Soldà},
      year={2024},
      eprint={2405.13842},
      archivePrefix={arXiv},
      primaryClass={math.LO},
      url={https://arxiv.org/abs/2405.13842}, 
}

@article{partial-impred-towsner,
 ISSN = {00224812, 19435886},
 URL = {http://www.jstor.org/stable/43303662},
 author = {Henry Towsner},
 journal = {The Journal of Symbolic Logic},
 number = {2},
 pages = {459--488},
 publisher = {[Association for Symbolic Logic, Cambridge University Press]},
 title = {PARTIAL IMPREDICATIVITY IN REVERSE MATHEMATICS},
 urldate = {2025-11-28},
 volume = {78},
 year = {2013}
}

@article{well-better-id-carroy-pequignot,
author = {Carroy, Raphaël and Pequignot, Yann},
year = {2014},
month = {10},
pages = {247-270},
title = {From Well to Better, the Space of Ideals},
volume = {227},
journal = {Fundamenta Mathematicae},
doi = {10.4064/fm227-3-2}
}

@article{id-higman-kabil-pouzet,
author = {Kabil, Mustapha and Pouzet, Maurice},
journal = {RAIRO - Theoretical Informatics and Applications - Informatique Théorique et Applications},
number = {5},
pages = {449-482},
publisher = {EDP-Sciences},
title = {Une extension d'un théorème de P. Jullien sur les âges de mots},
url = {http://eudml.org/doc/92428},
volume = {26},
year = {1992},
}

@article{freund-manca-wwo,
title={WEAK WELL ORDERS AND {F}raïssé’S CONJECTURE}, 
DOI={10.1017/jsl.2023.70}, 
journal={The Journal of Symbolic Logic}, 
author={Freund, Anton and Manca, Davide}, 
year={2023}, 
pages={1–16}
}

\end{document}